\documentclass[11pt,twoside]{article}
\usepackage{amsmath,amsthm,amssymb,mathabx}

\usepackage[a4paper ,left=20mm,right=20mm,top=30mm,bottom=30mm]{geometry}
\usepackage{imakeidx}
\usepackage{xr-hyper} 
\usepackage{xr}
\usepackage[all]{xy}
\usepackage[backref=page]{hyperref} 

\usepackage{authblk}
\usepackage{colortbl}

\usepackage{mathrsfs}

\usepackage{cleveref,fancyhdr}

\usepackage{seqsplit}

\usepackage{xstring}
\usepackage[xcdraw]{xcolor}
\definecolor{mycolor}{rgb}{0.122, 0.435, 0.698}

\hypersetup{
	colorlinks,
	citecolor=ocre,
	filecolor=ocre,
	linkcolor=ocre,
	urlcolor=ocre
}
\makeatletter
\DeclareOldFontCommand{\rm}{\normalfont\rmfamily}{\mathrm}
\DeclareOldFontCommand{\sf}{\normalfont\sffamily}{\mathsf}
\DeclareOldFontCommand{\tt}{\normalfont\ttfamily}{\mathtt}
\DeclareOldFontCommand{\bf}{\normalfont\bfseries}{\mathbf}
\DeclareOldFontCommand{\it}{\normalfont\itshape}{\mathit}
\DeclareOldFontCommand{\sl}{\normalfont\slshape}{\@nomath\sl}
\DeclareOldFontCommand{\sc}{\normalfont\scshape}{\@nomath\sc}
\makeatother

\usepackage[shortlabels]{enumitem}
\setlist{nolistsep} %

\newlist{asslist}{enumerate}{1} 
\setlist[asslist]{label=(\roman*), ref=\thethmT(\roman*)}
\crefalias{asslistenumi}{Assumption} 

\newlist{asslisttief}{enumerate}{1} 
\setlist[asslisttief]{label=(\roman*), ref=\thethmlistT(\roman*)}
\crefalias{asslisttiefenumi}{Property}

\newlist{thmlist}{enumerate}{1} 
\setlist[thmlist]{label=(\alph*), ref=\thethmT(\alph*)}

\usepackage{tikz} 

\usepackage{xcolor} 
\definecolor{ocre_old}{RGB}{243,102,25} 
\definecolor{ocre}{rgb}{0.122, 0.435, 0.698}
\usepackage{amsmath,amsfonts,amssymb,amsthm} 

\makeatletter
\newtheoremstyle{ocrenumbox}
{0pt}
{0pt}
{\sl}
{}
{\small\bf\sffamily\color{ocre}}
{\;}
{0.25em}
{\small\sffamily\color{ocre}\thmname{#1}\nobreakspace\thmnumber{\@ifnotempty{#1}{}\@upn{#2}}
	\thmnote{\nobreakspace\the\thm@notefont\sffamily\bfseries\color{black}---\nobreakspace#3.}} 

\newtheoremstyle{ocrenumhypbox}
{0pt}
{0pt}
{}
{}
{\small\bf\sffamily\color{ocre}}
{\;}
{0.25em}
{\small\sffamily\color{ocre}\thmname{#1}\nobreakspace\thmnumber{\@ifnotempty{#1}{}\@upn{#2}}
	\thmnote{\nobreakspace\the\thm@notefont\sffamily\bfseries\color{black}---\nobreakspace#3.}} 

\newtheoremstyle{blacknumex}
{5pt}
{5pt}
{\sl}
{} 
{\small\bf\sffamily}
{\;}
{0.25em}
{\small\sffamily{\tiny\ensuremath{\blacksquare}}\nobreakspace\thmname{#1}\nobreakspace\thmnumber{\@ifnotempty{#1}{}\@upn{#2}}
	\thmnote{\nobreakspace\the\thm@notefont\sffamily\bfseries---\nobreakspace#3.}}

\newtheoremstyle{blacknumbox} 
{0pt}
{0pt}
{\normalfont}
{}
{\small\bf\sffamily}
{\;}
{0.25em}
{\small\sffamily\thmname{#1}\nobreakspace\thmnumber{\@ifnotempty{#1}{}\@upn{#2}}
	\thmnote{\nobreakspace\the\thm@notefont\sffamily\bfseries---\nobreakspace#3.}}

\newtheoremstyle{ocrenum}
{5pt}
{5pt}
{\sl}
{}
{\small\bf\sffamily\color{ocre}}
{\;}
{0.25em}
{\small\sffamily\color{ocre}\thmname{#1}\nobreakspace\thmnumber{\@ifnotempty{#1}{}\@upn{#2}}
	\thmnote{\nobreakspace\the\thm@notefont\sffamily\bfseries\color{black}---\nobreakspace#3.}} 
\makeatother

\theoremstyle{ocrenumbox}
\newtheorem{thmT}{Theorem}[section]
\newtheorem{theoT}{Theorem}
\newtheorem{theoremeT}[thmT]{Theorem}
\newtheorem{lemT}[thmT]{Lemma}
\newtheorem{corT}[theoT]{Corollary}

\theoremstyle{ocrenumhypbox}
\newtheorem{hypT}[thmT]{Hypothesis}
\theoremstyle{blacknumex}

\theoremstyle{blacknumbox}
\newtheorem{definitionT}[thmT]{Definition}
\newtheorem{notationT}[thmT]{Notation}

\newtheorem{conditionT}[thmT]{Condition}

\theoremstyle{ocrenum}
\newtheorem{proposition}[thmT]{Proposition}
\newtheorem{propT}[thmT]{Proposition}

\newtheorem{corollaryT}[thmT]{Corollary}

\RequirePackage[framemethod=default]{mdframed} 

\newmdenv[skipabove=7pt,
skipbelow=7pt,
backgroundcolor=black!5,
linecolor=ocre,
innerleftmargin=5pt,
innerrightmargin=5pt,
innertopmargin=5pt,
leftmargin=0cm,
rightmargin=0cm,
innerbottommargin=5pt]{tBox}

\newmdenv[skipabove=7pt,
skipbelow=7pt,
rightline=false,
leftline=true,
topline=false,
bottomline=false,
backgroundcolor=ocre!10,
linecolor=ocre,
innerleftmargin=5pt,
innerrightmargin=5pt,
innertopmargin=5pt,
innerbottommargin=5pt,
leftmargin=0cm,
rightmargin=0cm,
linewidth=4pt]{eBox}	

\newmdenv[skipabove=7pt,
skipbelow=7pt,
rightline=false,
leftline=true,
topline=false,
bottomline=false,
linecolor=ocre,
innerleftmargin=5pt,
innerrightmargin=5pt,
innertopmargin=0pt,
leftmargin=0cm,
rightmargin=0cm,
linewidth=4pt,
innerbottommargin=0pt]{dBox}	

\newmdenv[skipabove=7pt,
skipbelow=7pt,
rightline=false,
leftline=true,
topline=false,
bottomline=false,
linecolor=gray,
backgroundcolor=black!5,
innerleftmargin=5pt,
innerrightmargin=5pt,
innertopmargin=5pt,
leftmargin=0cm,
rightmargin=0cm,
linewidth=4pt,
innerbottommargin=5pt]{cBox}

\newenvironment{theorem}{\begin{tBox}\begin{theoremeT}}{\end{theoremeT}\end{tBox}}

\newenvironment{theo}{\begin{tBox}\begin{theoT}}{\end{theoT}\end{tBox}}
\newenvironment{coralph}{\begin{tBox}\begin{corT}}{\end{corT}\end{tBox}}

\newenvironment{defi}{\begin{dBox}\begin{definitionT}}{\end{definitionT}\end{dBox}}	
\newenvironment{notation}{\begin{dBox}\begin{notationT}}{\end{notationT}\end{dBox}}	
\newenvironment{condition}{\begin{dBox}\begin{conditionT}}{\end{conditionT}\end{dBox}}	
\newenvironment{lem}{\begin{dBox}\begin{lemT}}{\end{lemT}\end{dBox}}	
\newenvironment{prop}{\begin{dBox}\begin{propT}}{\end{propT}\end{dBox}}
		
\newenvironment{corollary}{\begin{dBox}\begin{corollaryT}}{\end{corollaryT}\end{dBox}}	
\newenvironment{cor}{\begin{dBox}\begin{corollaryT}}{\end{corollaryT}\end{dBox}}	

\newenvironment{remark}{\par\vspace{10pt}\small 
	\begin{list}{}{
			\leftmargin=35pt 
			\rightmargin=25pt}\item\ignorespaces 
		\makebox[-2.5pt]{\begin{tikzpicture}[overlay]
				\node[draw=ocre!60,line width=1pt,circle,fill=ocre!25,font=\sffamily\bfseries,inner sep=2pt,outer sep=0pt] at (-15pt,0pt){\textcolor{ocre}{R}};\end{tikzpicture}} 
		\advance\baselineskip -1pt}{\end{list}\vskip5pt} 

\makeatletter
\renewcommand{\@seccntformat}[1]{\llap{\textcolor{ocre}{\csname the#1\endcsname}\hspace{1em}}} 
\renewcommand{\section}{\@startsection{section}{1}{\z@}
	{-4ex \@plus -1ex \@minus -.4ex}
	{1ex \@plus.2ex }
	{\normalfont\large \bf \color{ocre}}}
\renewcommand{\subsection}{\@startsection {subsection}{2}{\z@}
	{-3ex \@plus -0.1ex \@minus -.4ex}
	{0.5ex \@plus.2ex }
	{\normalfont\large\bf\color{ocre} }}

\newcommand{\EGF}{{E(\GF)}}
\makeatletter
\def\author#1{\gdef\autrun{\def\and{\unskip, }#1}\gdef\@author{#1}}

\makeatother

\newcommand{\tw}[1]{{}^#1\!}

\theoremstyle{definition}

\newtheorem{rem}[thmT]{Remark}

{\color{ocre}}
\theoremstyle{plain}

\theoremstyle{definition}

\numberwithin{equation}{section}
\numberwithin{table}{section}

\def\norm#1#2{{\operatorname N}_{#1}(#2)}

\newcommand{\Id}{\operatorname {Id}}

\newcommand{\wt}{\widetilde}
\newcommand{\wc}{\widecheck}
\newcommand{\wh}{\widehat }
\newcommand{\bN}{{\mathbf N}}
\newcommand{\Levi}{{\mathbf L}}

\newcommand{\wG}{{\widetilde G}}

\newcommand{\bC}{{\mathbf C}}
\newcommand{\la}{\ensuremath{\lambda}}

\newcommand{\bT}{{\mathbf T}}

\newcommand{\bB}{{\mathbf B}}
\newcommand{\bP}{{\mathbf P}}

\newcommand{\bG}{{{\mathbf G}}}

\newcommand{\bH}{{{\mathbf H}}}

\newcommand{\bL}{{{\mathbf L}}}

\newcommand{\bM}{{\mathbf M}}
\newcommand{\bS}{{\mathbf S}}
\newcommand{\HF}{{{\bH}^F}}

\newcommand{\wbG}{\wt{\mathbf G}}

\newcommand{\Irr}{\mathrm{Irr}}
\newcommand{\irr}{\mathrm{Irr}}

\newcommand\lie\mathfrak

\newcommand{\lcm}{\mathrm{lcm}}
\newcommand{\Irrl}{\mathrm{Irr_{\ell'}}}

\newcommand{\Gal}{\mathrm{Gal}}

\newcommand{\Res}{\operatorname{Res}}
\newcommand{\Ind}{\operatorname{Ind}}
\newcommand{\C}{\mathrm{C}}
\newcommand{\N}{\mathrm{N}}
\newcommand{\Lang}{\mathscr{L}}

\newcommand{\SL}{\operatorname{SL}}
\newcommand{\SU}{\operatorname{SU}}
\newcommand{\GU}{\operatorname{GU}}
\newcommand{\GL}{\operatorname{GL}}

\newcommand{\PGL}{\operatorname{PGL}}

\newcommand{\ZZ}{\ensuremath{\mathbb{Z}}}
\newcommand{\CC}{\ensuremath{\mathbb{C}}}

\newcommand{\ov}{\overline }

\newcommand{\Cent}{\ensuremath{{\rm{C}}}}
\newcommand{\NNN}{\ensuremath{{\mathrm{N}}}}

\newcommand{\Sym}{{\mathcal{S}}}

\def\restr#1|#2{\left.#1\right\rceil_{#2}}

\usepackage{imakeidx}
\makeindex[columns=2,intoc]

\def\II#1@#2{\index{#1@$#2$}{{\color{ocre}#2}}}

\newcommand{\E}{\ensuremath{\operatornams{E}}}

\newcommand{\QQ}{\ensuremath{\mathbb{Q}}}
\newcommand{\tD}{\ensuremath{\mathrm{D}}}
\newcommand{\Cy}{\mathrm C}
\newcommand{\tC}{\mathrm C}
\newcommand{\tB}{\mathrm B}

\newcommand{\cE}{\mathcal E}

\newcommand{\cH}{\mathcal H}
\newcommand{\galh}{\mathcal H}
\newcommand{\galhl}{{{\mathcal H}_{\ell}}}

\newcommand{\cX}{\mathcal X}
\newcommand{\cV}{\mathcal V}
\newcommand{\rC}{\mathrm C}

\newcommand{\rZ}{\mathrm Z}
\newcommand{\gal}{\mathcal G}

\newcommand{\calC}{\mathcal C}

\newcommand{\calG}{\ensuremath{\mathcal G}}

\newcommand{\calH}{\ensuremath{\mathcal H}}

\newcommand{\al}{{\alpha}}

\newcommand{\UCh}{\operatorname{UCh}}

\newcommand{\FF}{{\mathbb{F}}}
\newcommand{\GG}{{\mathbb{G}}}

\newcommand{\si}{\ensuremath{\sigma}}

\newcommand{\GF}{{{\bG^F}}}

\newcommand{\wGF}{{{{\wbG}^F}}}

\newcommand{\subG}{\mathfrak{G}}
\newcommand{\subM}{\mathfrak{M}}

\makeatletter
\def\Set#1{\Set@h#1@}
\def\Lset#1{\Lset@h#1@}
\def\Set@h#1|#2@{\left\{\left.#1\vphantom{#2}\hskip.1em\,\right\mid \,\relax #2\right\}}
\def\Lset@h#1@{\left\{#1\right\}}
\def\CALC#1{\CALC@h#1@}
\def\CALC@h#1|#2@{\calC^{#1}(#2)}
\def\CALCrad#1{\CALCrad@h#1@}
\def\CALCrad@h#1|#2@{\calC_\radic^{#1}(#2)}

\def\CALCNC#1{\CALCNC@h#1@}
\def\CALCNC@h#1|#2@{\calC_{\radic,nc}^{#1}(#2)}

\def\restr#1|#2{\left.#1\right\rceil_{#2}}
\def\spann<#1>{\left\langle#1\right\rangle}
\def\spa#1{\left\langle#1\right\rangle}

\def\Spann<#1>{\Spann@h#1@}
\def\Spann@h#1|#2@{\left\langle\left.#1\vphantom{#2}\hskip.1em\right.\mid\relax #2 \right\rangle}
\def\Set#1{\Set@h#1@}
\def\Set@h#1|#2@{\left\{\left.#1\vphantom{#2}\hskip.1em\,\right.
	\mid\relax #2\right\}}
\def\set#1{\set@h#1@}
\def\set@h#1@{\left\{#1\right\}}
\def\spann<#1>{\left\langle#1\right\rangle}
\makeatother

\newcommand{\Aut}{\mathrm{Aut}}

\newcommand{\Z}{\operatorname Z}
\newcommand{\Sp}{\operatorname{Sp}}
\newcommand{\calP}{\mathcal P}

\newcommand{\forevery}{{\text{\quad\quad for every }}}
\newcommand{\und}{{\text{ and }}}

\newcommand{\G}{{{{G}}}}

\renewcommand{\E}{\operatorname{E}}

\newcommand{\lra}{\longrightarrow}
\newcommand{\tE}{\mathrm E}
\newcommand{\tA}{\mathrm A}
\newcommand{\tF}{\mathrm F}
\newcommand{\tG}{\mathrm G}

\newcommand{\cC}{\mathcal C}

\newcommand{\Norm}{\operatorname{N}}

\newcommand{\ab}{\mathrm{ab}}

\DeclareMathOperator{\Lin}{Lin}

\newcommand{\wrt}{{with respect to\ }}

\title{Towards the inductive McKay--Navarro Condition for 
groups of Lie type 
}
\date{\today}

\author{Lucas Ruhstorfer, A. A. Schaeffer Fry,  Britta Sp{\"a}th, and Jay Taylor
		\thanks{The authors gratefully acknowledge support from a SQuaRE at the American Institute of Mathematics. The second-named author  gratefully acknowledges support from the National Science Foundation, Award No. DMS-2100912. 
         Some of the research of the third- and fourth-named authors was conducted in the framework of the research training group \textit{GRK 2240:
        Algebro-Geometric Methods in Algebra, Arithmetic and Topology}, funded by the DFG.}}

\newcommand{\Addresses}{{
  \bigskip
  \footnotesize

L.~Ruhstorfer, \textsc{School of Mathematics and Natural Sciences, University of Wuppertal, Gaußstr. 20,
		42119 Wuppertal, Germany}
\par\nopagebreak
  \textit{E-mail}: \texttt{ruhstorfer@uni-wuppertal.de}

\medskip
  A.~A.~Schaeffer Fry, \textsc{Department of Mathematics, University of Denver, Denver, CO 80210, USA}\par\nopagebreak
  \textit{E-mail}: \texttt{mandi.schaefferfry@du.edu}

  \medskip

  B.~Sp{\"a}th, \textsc{School of Mathematics and Natural Sciences, University of Wuppertal, Gaußstr. 20,
		42119 Wuppertal, Germany} \par\nopagebreak
  \textit{E-mail}: \texttt{bspaeth@uni-wuppertal.de}

  \medskip
  
J.~Taylor, \textsc{Department of Mathematics, The University of Manchester, Oxford Road, Manchester, M13 9PL, UK}
\par\nopagebreak
  \textit{E-mail}: \texttt{jay.taylor@manchester.ac.uk}

}}
\begin{document} 
	\maketitle
	\abstract{We gather tools for proving the inductive McKay--Navarro (or Galois--McKay) condition for groups of Lie type and odd primes. We use this to establish a bijection in the case of quasisimple groups of Lie type $\tA$ satisfying the equivariance properties needed for the condition. We also prove the inductive conditions for the subset of unipotent characters.}

\tableofcontents

\section{Introduction}
Since the time of Brauer, the study of the representation theory of a finite group through its relationship to local data has been a leading objective in the theory of finite groups.  
One of the most influential conjectures recently in this realm has been the McKay conjecture, along with its refinements. 
The McKay conjecture (recently announced as a theorem due to the third author and M. Cabanes \cite{CS24}) says that for a prime $\ell$ dividing the order of a finite group $G$, the number of irreducible characters of $\ell'$-degree, $\irr_{\ell'}(G)$, should be the same as that of $\Norm_G(P)$ for $P$ a Sylow $\ell$-subgroup of $G$. 

In this paper, we are concerned with the McKay--Navarro conjecture (also known as the Galois--McKay conjecture, \cite[Conjecture A]{nav04}), which suggests that a bijection between $\irr_{\ell'}(G)$
 and $\irr_{\ell'}(\Norm_G(P))$ can be further chosen to be equivariant with respect to certain Galois automorphisms 
 $\cH_\ell\leq \mathrm{Gal}(\QQ^{\mathrm{ab}}/\QQ)$. 
(See Definition \ref{defH} below.) Here, $\QQ^{\mathrm{ab}}$ denotes the extension of $\QQ$ generated by all roots of unity in $\CC$.

The McKay--Navarro conjecture was reduced by Navarro--Sp{\"a}th--Vallejo in \cite{NSV20} to the so-called ``inductive Galois--McKay" conditions for quasisimple groups. We will refer to these as the inductive McKay--Navarro (iMN) conditions. These iMN conditions build upon the inductive McKay conditions of Isaacs--Malle--Navarro in \cite{IMN}, adding $\galh_\ell$-compatibility. The first-named author
 completed these conditions when $G$ is a group of Lie type in defining characteristic in  \cite{R21}, and he and the second-named author completed the conditions when $\ell=2$ in \cite{SF22, RSF22, RSF22b}. 
 
The goal of this paper is to lay some foundations for dealing with the remaining cases -  namely, the case of groups of Lie type defined in characteristic $p\neq \ell$ when $\ell$ is odd.
We establish several results in this direction, which will be useful in the pursuit of the iMN conditions in this situation.
In particular, we introduce criterion in Section \ref{sec:criterion}  that will be useful for proving that groups of Lie type satisfy the iMN conditions - see Theorem \ref{crit_iGalMcK} and Corollary \ref{12vib} below.

 The iMN conditions can be roughly described as an ``equivariance" condition and an ``extension" condition.  In Sections \ref{sec:equivbij}--\ref{sec:transversallocal}, we primarily work with the  equivariance condition. Namely, this is the first part of \cite[Definition 3.1]{NSV20}, which we state here:

\begin{condition}\label{cond:equivcondition}
Let $\ell$ be a prime.
Let $G$ be a finite quasisimple group with $\ell\mid|G|$ and let $P\in\mathrm{Syl}_\ell(G)$.  
Then there is a proper $\Aut(G)_P$-stable subgroup $M$ of $G$ with $\Norm_G(P)\leqslant M$ and an 
$\Aut(G)_P\times\galh_\ell$-equivariant bijection 
$\irr_{\ell'}(G)\rightarrow\irr_{\ell'}(M)$ such that corresponding characters lie over the same character of $\Z(G)$.
\end{condition}

In Section \ref{sec:criterion}, we prove the following.

\begin{theo}\label{thm:critintro}
Let $G$ be the universal covering group of a nonabelian simple group and let $\ell\mid |G|$ be an odd prime. Assume that Conditions (i)--(vi) of Theorem \ref{crit_iGalMcK}  hold for $\galh=\galh_\ell$, $\subG=\Irr_{\ell'}(G)$ and $\subM=\Irr_{\ell'}(M)$, where $M\lneq G$ is a subgroup containing $\norm{G}{P}$ for $P$ a Sylow $\ell$-subgroup of $G$. Then Condition \ref{cond:equivcondition} holds for $G$ and $M$. 

If (vii) of Corollary \ref{12vib} additionally holds, then the inductive McKay--Navarro conditions \cite[Def.~3.1]{NSV20} hold for $G$ and the prime $\ell$.
\end{theo}

Our next main result is that Condition \ref{cond:equivcondition} holds for a group of Lie type $G$ defined in characteristic $p\neq \ell$ if certain local extension conditions hold. 
 
\begin{theo}\label{thm:thmA}
Assume $(\bG, F)$ is a finite reductive group such that $\bG$ is of simply connected type, and fix a Sylow $d$-torus $\bS$ of $(\bG,F)$, where $d=d_\ell(q)$, and a regular embedding $\bG\hookrightarrow\wt\bG$.

Denote by $\wt N:=\Norm_\wGF(\bS)$ and $N:=\wt N\cap \GF=\Norm_\GF(\bS)$ the normalizers of $\bS$ and set
\begin{align*}
\mathfrak{C}:= \{ \chi\in  \Irr(\Cent_\wGF(\bS))\mid  \Irr(N\mid \restr \chi|{\Cent_\GF(\bS)}) \cap \Irrl(N)\neq \emptyset \}.
\end{align*} 
Assume there exists some $\galh\ltimes (\Irr(\wt N/N)\rtimes  (\GF\EGF)_\bS
	)$-equivariant extension map $\wt \Lambda$ for $\mathfrak{C}$ with respect to $\Cent_\wGF(\bS)\lhd \wt{N}$, where $\galh\leq\galh_\ell$ is a subgroup of Galois automorphisms and $E(\bG^F)$ is the group of field and graph automorphisms as in Section~\ref{sec:bijpfthmA}.

Then there exists some  $\galh\ltimes (\Irr(\wGF/\GF)\rtimes (\GF\EGF)_{\bS})$-equivariant bijection
\begin{equation*}
\wt \Omega: \Irr(\wGF\mid \Irrl(\GF ))\lra\Irr(\wt N\mid \Irrl(N )),
\end{equation*}
such that 
$\Irr\left(\Z(\wt\bG^F)\mid \Omega(\chi)\right)=\Irr\left({\Z(\wt\bG^F)}\mid \chi\right)$
for every $\chi \in \Irr(\wGF\mid \Irrl(\GF ))$.
\end{theo}

 The proof of Theorem \ref{thm:thmA} is the goal of Section \ref{sec:equivbij}.
It turns out that it will sometimes be useful to use the result of Theorem \ref{thm:thmA} with respect to a different overgroup. In this case, Theorem \ref{thm:notbutterfly} will allow us to replace $\wt{\bG}^F$ with such an overgroup.

 In Section \ref{sec:transversal}, we use the theory of Generalized Gelfand Graev Representations to prove that there is a convenient transversal in the case of groups of type $\tA$ (see Theorem \ref{thm:typeAtransversalglobal}). We use this to prove that the equivariance condition holds for the groups $\SL_n(q)$ and $\SU_n(q)$, together with Theorem \ref{thm:thmA} and the work of Petschick and the third-named author \cite{SoniaSpaeth}, who have proved the necessary local conditions in these cases. 
 Along the way, we present candidates for the criterion in Theorem \ref{crit_iGalMcK} and Corollary \ref{12vib} in the case of type $\tA$, showing that the conditions in Theorem \ref{crit_iGalMcK} are satisfied.

  \begin{coralph}\label{cor:equivcondtypeA}
Let $G=\bG^F\in\{\SU_n(q), \SL_n(q)\}$ and let $\ell\nmid q$ be an odd prime dividing $|G|$. Then the conditions of Theorem \ref{crit_iGalMcK} are satisfied for $\galh:=\galh_\ell$, $\subG:=\Irr_{\ell'}(G)$, and $\subM:=\Irr_{\ell'}(M)$, taking $M:=\norm{G}{\bS}$ for $\bS$ a Sylow $d_\ell(q)$-torus of $(\bG, F)$, $\widecheck{G}:=\Lang^{-1}(\rZ(\bG))$, and  $B$ as in Corollary \ref{cor:Bglobal}. 

In particular, $G$ satisfies Condition \ref{cond:equivcondition}.
\end{coralph}

We remark that we anticipate the last condition of the criterion, namely the condition stated in Corollary \ref{12vib}, will also be satisfied in the situation of Corollary \ref{cor:equivcondtypeA}. (The first part of the additional condition, namely that $B$ normalizes $\wh M$, is satisfied in this situation - see Corollary \ref{cor:Bglobal}.) This would also ensure the ``extension" part (hence the full iMN condition) for type $\tA$, and will be the objective of future work by the authors.

In Sections \ref{sec:extunip1} and \ref{sec:extunip2}, we begin treating the ``extension" condition. This condition requires that the bijection from Condition \ref{cond:equivcondition} also behaves well when taking projective extensions of the characters to the stabilizer in the automorphism group. In the case of unipotent characters, we obtain the full iMN condition.

\begin{theo}\label{thm:unipcriterion}
Let $G=\bG^F$ with $\bG$ a connected reductive group of simply connected type and $F\colon \bG\rightarrow\bG$ a Frobenius endomorphism and suppose that $\bG$ is not of type $\tB_2,\tF_4$ if $p=2$ and $\tG_2$ if $p=3$ (so that $G$ admits no exceptional Suzuki or Ree automorphism). Let $\ell$ be an odd prime and let $\galh:=\galh_\ell$. Then $G$ satisfies the conclusion of Corollary \ref{12vib} for $\subG:=\Irr_{\ell'}(G)\cap\UCh(G)$ and $\subM:=\Omega(\subG)$, where $M:=\norm{G}{\bS}$ is the normalizer of a Sylow $d_\ell(q)$-torus and $\Omega$ is the McKay bijection given in \cite{Ma07}, and where we take $\widecheck{G}:=\wt{G}$ to be the group $\wt{G}:=\wt\bG^F$ obtained by a regular embedding $\bG\rightarrow \wt\bG$.

 That is, there exists a $\widecheck M\wh M\times\galh$-equivariant bijection $\Omega:\subG\lra\subM$ such that 
    $$((\widecheck G E)_{\chi^{\calH}}, G, \chi)_\cH \geq_c ((\widecheck{M}\wh M)_{\psi^\cH}, M, \psi)_{\cH}$$  for every $\chi\in  \subG$ and $\psi:=\Omega(\chi)$.
\end{theo}

\section{Preliminaries on the Galois automorphisms}
	
	We begin with some number-theoretic observations regarding the Galois automorphisms under consideration.
	
\begin{notation}\label{defH} As defined by Gabriel Navarro \cite{nav04}, we let $\cH_\ell$ denote the subgroup of the Galois group $\gal:=\Gal(\QQ^{\ab}/\QQ)$ comprised of those $\sigma\in\gal$ satisfying that there is some $e\geq 0$ such that  $\sigma$ acts on $\ell'$-roots of unity by $\zeta\mapsto \zeta^{\ell^e}$. 
 
For an integer $d\geq1$, we will also let $\cH_{[d]}$ denote the subgroup of $\gal$ stabilizing any $d$-th root of unity. 
\end{notation}

It will sometimes be useful to note that $\galh_\ell$ can be identified with the Galois group $\mathrm{Gal}(\QQ_\ell^{\mathrm{ab}}/ \QQ_\ell)$. As such, we observe the following immediate consequence of Hensel's Lifting Lemma, which could also be seen using Gauss sums and arguments as in \cite[Section 4]{SFT23}.

\begin{lem}\label{lem:square}
Let $\ell$ be an odd prime and let $a$ be an integer relatively prime to $\ell$ that  is a square modulo $\ell$. Then $\sqrt{a}^\sigma=\sqrt{a}$ for all $\sigma\in\galh_\ell$.
\end{lem}

Throughout, for an integer $q$ and odd prime $\ell$, we let $d_\ell(q)$ denote the multiplicative order of $q$ modulo $\ell$.  
\begin{lem} 
\label{lem:numtheory}
	Let $\ell$ be an odd prime and $d$ any integer dividing  $(\ell-1)$. (For  example, this is the case when $d$ is the order of an element of $\FF_\ell^\times$.) Then: 
	\begin{enumerate} 
	\item \label{H} $\cH_{\ell}\subseteq \cH_{[d]}$; 
	\item \label{2mH} if $d=2m$ with $m$ odd, then $\cH_{[d]}=\cH_{[m]}$; 
	\item \label{2msqrt} if $d=2m$ with $m$ odd and $q$ is an  integer with $d_\ell(q)=d$, then $\sqrt{-q}^\sigma=\sqrt{-q}$ for any $\sigma\in\galh_\ell$;
	\item \label{sqrt} if $d$ is odd and $q$ is an integer with $d_\ell(q)=d$, then $\sqrt{q}^\sigma=\sqrt{q}$ for any $\sigma\in\galh_\ell$.
	\end{enumerate}
	\end{lem}
\begin{proof}
If $\xi$ is a $d$th root of unity (and hence an $(\ell-1)$st root of unity), then $\sigma(\xi)=\xi^{\ell^e}=\xi$ for any $\sigma\in\galh_\ell$, so \ref{H} holds. For \ref{2mH}, note that any member of $\gal$ must necessarily fix $-1$, and that any $d$th root of unity must be of the form $\pm\zeta_m$ with $\zeta_m$ an $m$th root of unity.
In the situation of \ref{2msqrt}, we have $\ell$ divides $q^{m}+1$, and hence $q^{m}\equiv -1\pmod \ell$. Since $m$ is odd, we see $q$ is a square in $\FF_\ell^\times$ if and only if $-1$ is. This means that  $-q$ is a square in $\FF_\ell^\times$, and we apply Lemma \ref{lem:square}.

For \ref{sqrt}, note that we may assume that $q$ is not a square, since otherwise the statement is clear. Hence $q$ is an odd power of a prime $p$. By construction, $d$ is the order of $q$ in $\FF_\ell^\times$. As $\ell-1$ is even and $d$ is odd, this shows that  $q$ is a square in $\FF_\ell^\times$, and we again apply Lemma \ref{lem:square}.
\end{proof}

 \section{Criterion for the inductive McKay--Navarro condition}\label{sec:criterion}

In this section, we introduce a statement that simplifies checking the inductive McKay--Navarro condition. The statement below (\Cref{crit_iGalMcK_full}) is modeled after   \cite[Thm.~2.12]{S12}, which is a criterion for verifying the inductive McKay condition and was applied to simple groups of Lie type in various cases.

\subsection{The inductive McKay--Navarro condition}

Before stating our criterion, we introduce the inductive conditions from \cite{NSV20}. We begin with some notation.

\begin{defi}\label{def:extmap}
Let $G\lhd A$ and  $\chi\in\Irr(G)$ invariant under $A$. Then $(A,G,\chi)$ is an ordinary character triple and there exists a projective representation $\calP$ of $A$ associated to $\chi$, see \cite[Def.~5.2]{Navarro_book}.
Assume that $H\leq A$ with $A=GH$ with $\Cent_A(G)\leq H$. If $\psi\in\Irr(M)$ is some $H$-invariant character for $M:=G\cap H$, then there exists some projective representation $\calP'$ of $H$ associated to $\psi$. We say that $(\calP,\calP')$ gives
\[(A,G,\phi)\geq_c(H,M,\psi),\]
whenever the factor set of $\calP$ restricted to $H\times H$ is the factor set of $\calP'$ and for each $c\in\Cent_A(G)$ the matrices $\calP(c)$ and $\calP'(c)$ are scalar matrices associated to the same scalar; see also \cite[Def.~10.14]{Navarro_book}. We also write $(A,G,\phi)\geq_c(H,M,\psi)$ if there exist $(\calP,\calP')$ giving that relation.
\end{defi}

For the inductive McKay--Navarro condition, we need adaptions of the above notions.

\begin{defi} Given $G\lhd A$, the group $\gal\times \Aut(A)_G$ acts on $\Irr(G)$. 
Let $\chi\in\Irr(G)$ and $\alpha \in \gal\times \Aut(A)_G$ with $\chi^\al =\chi$. 
If $\chi$ extends to a character $\wt \chi$ of $A$, then there exists a unique linear character $\mu\in \Irr(A)$ with 
        $G\leq \ker(\mu)$ and 
        \[\wt\chi^\alpha=\wt\chi\mu\] 
    	by Gallagher's theorem \cite[Cor.~6.17]{Isa}. We will write $[\wt\chi, \alpha]:=\mu$ in this situation.
The projective representation $\calP^\alpha$ is again a projective representation of $A$ associated to $\chi$ and hence similar to $\mu \calP$, i.e. $\calP^\al\sim \mu \calP$ for a unique map $\mu: A/G\lra \CC^\times$. We denote $\mu$ again by $[\calP,\al]$, see also \cite[Lem.~1.4, Lem.~2.3]{NSV20}. 

Given $\galh\leq \calG$, $G\lhd A$ and $\chi\in \Irr(G)$ we denote by $\chi^\galh$ the sum of all $\galh$-conjugates of $\chi$ and call $(A,G,\chi)_\cH$ an $\galh$--triple if $\chi^{\cH}$ is $A$-invariant, see \cite[\S 1]{NSV20}. Assume that $H\leq A$ with $A=GH$ and $\Cent_A(G)\leq H$. Let $\psi\in\Irr(M)$ such that $(H,M,\psi)_\galh$ is an $\galh$--triple. 
Assume $(A_\chi,G,\chi) \geq_c 
(H_\psi,M,\psi) $ given by $(\calP,\calP')$. 
If $(H\times \galh)_\chi=(H\times \galh)_\psi$ and $(\calP,\calP')$ can be chosen such that they satisfy 
\[  [\calP,\al](h)=[\calP',\al](h) \forevery h\in H_\psi, \]
then we write 
\[ 
(A,G,\phi)_\galh\geq_c(H,M,\psi)_\galh ,\]
see  \cite[Def.~1.5]{NSV20}.\end{defi}

Now the \textit {inductive McKay--Navarro (iMN) condition} for a given simple group $S$ and a prime $\ell$ is satisfied if for the universal covering group $G$ of $S$, any Sylow $\ell$-subgroup $P$ of $G$ and $\Delta=\Aut(G)_P$, there exists some $\Delta$-stable subgroup $M$ of $G$ with $\NNN_G(P)\leq M\lneq G$ and an 
$(\galh_\ell\times \Delta)$-equivariant bijection $\Omega: \Irrl(G)\lra \Irrl(M)$ such that 
\[ 
(G\rtimes \Delta_{\chi^{\galh_\ell}}, G,\chi)_\galhl\geq 
(M\rtimes \Delta_{\Omega(\chi)^{\galh_\ell}}, M,\Omega(\chi))_\galhl.\]

That is, the inductive McKay--Navarro conditions essentially require Condition \ref{cond:equivcondition}, together with the above character triple condition (which we sometimes call the ``extension condition").

\subsection{A criterion for the iMN condition}

As with the criterion in \cite[Thm.~2.12]{S12}, our criterion will require the use of extension maps, but now requiring compatibility with Galois automorphisms.

\begin{defi}
For $G\lhd A$ and subsets $\subG\subseteq \Irr(G)$ and $\mathfrak{A}\subseteq\Irr(A)$, we write $\Irr(A\mid\subG)$ for the irreducible characters of $A$ whose irreducible constituents on restriction to $G$ lie in $\subG$ and $\Irr(G\mid\mathfrak{A})$ for the irreducible constituents of characters in $\mathfrak{A}$ on restriction to $G$. Further, an \textit{extension map} \wrt $G\lhd A$ for $\subG$ is a map $\Lambda:\subG\lra \bigcup_{G\leq I \leq A} \Irr(I)$ such that for every $\chi\in\mathfrak{G}$, $\Lambda(\chi)$ is an extension of $\chi$ to $A_\chi$. If an extension map \wrt $G\lhd A$ for $\subG$ exists, we say that \textit{maximal extendibility holds \wrt $G\lhd A$ for $\subG$}, see \cite[4.1]{S09} and \cite[4.3]{S10a}.

Given an extension map $\Lambda$ and a subgroup $\galh\leq\gal$, we say that $\Lambda$ is \emph{$\galh$-equivariant} if $\Lambda(\la^\si)=\Lambda(\la)^\si$ for every $\si\in \cH$. 
\end{defi}

    We also note that a subset of linear characters in $\Irr(G)$, which is at the same time a group given by multiplication, acts on $\Irr(G)$ by multiplication.

We are now ready to introduce our criterion.
The following statement helps to construct bijections as required in the inductive McKay--Navarro condition in the case where $G$ is a group of Lie type. 
\begin{theorem}\label{crit_iGalMcK}
Let $G$ be a finite group, $\ell$ a prime with $\ell\mid |G|$, $Q$ a Sylow $\ell$--subgroup of $G$ and $M\lneq G$ a proper subgroup. Let $\galh\leq\gal$ be some subgroup.
Additionally, let $\subG\subseteq \Irr(G)$ and $\subM\subseteq \Irr(M)$.  Assume all of the following: 
\begin{asslist}
    \item There exists a group $\widecheck{G}$ with $G\lhd \widecheck{G}$ and a finite group $E$ such that $\widecheck{G}\rtimes E$ is defined, $\widecheck{G}/G$ is abelian, $\Cent_{\widecheck{G}\rtimes E}(G)=\Z(\widecheck{G})\Cent_{E}(G)$;  $G\lhd \widecheck{G}\rtimes E$; and $\subG$ and $\subM$ are $\galh\times \norm{\widecheck{G}}{M}$-stable.   
    \item The subgroup $M$ is $(\widecheck{G}\rtimes E)_Q$-stable and  $\NNN_G(Q)\leq M\lneq G$.
    \item \label{hauptprop_maxext} Maximal extendibility holds with respect to  $G\lhd \widecheck{G}$ and $M \lhd \widecheck M:=\norm{\widecheck{G}}{M}$ for $\subG$ and $\subM$, respectively.
    \item \label{glo-*}\label{1.10iv}
    For $\wh M:=\NNN_{GE}(M)$ there exist some $\wh M$-stable $\widecheck M$-transversals $\subG_0$ in $\subG$ and $\subM_0$ in $\subM$. Additionally there exist  $\wh M$-equivariant extension maps $\Phi_{glo}$ and $\Phi_{loc}$ \wrt $G\lhd GE$ for $\subG_0$ and \wrt $M\lhd \wh M$ for $\subM_0$, such that 
    \[ \Irr( \Cent_E(G) \mid  \Phi_{glo}(\chi)) =  \Irr( \Cent_E(G)\mid \Phi_{loc}(\psi) ) 
 \text{ for every }\chi\in\subG_0 \und \psi\in\subM_0.\]
 \item \label{assvi} 
    
    There exists a subgroup $B\leq (\widecheck{G} E \times \galh)$ with
    \[  \NNN_{\widecheck{G} E}(Q) B=\NNN_{\widecheck{G} E}(Q) \times \galh,\]
    that stabilizes $\subG_0$ and $\subM_0$.
    \item \label{crit_bij}
     For $\widecheck\subM:=\Irr(\widecheck M\mid\subM)$ and $\widecheck\subG:=\Irr(\widecheck{G}\mid \subG)$ there exists a $\Irr(\widecheck M/M)\rtimes(\wh M B )$-equivariant bijection $\widecheck \Omega:\widecheck\subG\lra \widecheck\subM$ with 
    $ \widecheck\Omega(\widecheck \subG\cap\Irr(\widecheck{G}\mid \nu))= \widecheck \subM\cap\Irr(\widecheck M\mid \nu)$ for every $\nu\in\Irr(\Z(\widecheck{G}))$.
 
    \end{asslist}

    \noindent Then there exists a $\wh M\widecheck M\times\galh$-equivariant bijection $\Omega:\subG\lra\subM$ such that 
    \begin{align}
    \label{geqcHct}
    ((\widecheck G E)_{\chi}, G, \chi)& \geq_c ((\widecheck{M}\wh M)_{\psi}, M, \psi)  \text{ for every }\chi\in  \subG \und \psi:=\Omega(\chi).
    \end{align}
    \end{theorem}

 \begin{proof}
    	       
        As $\subG_0$ is an $E$-stable $\wc G$-transversal, we see that $(\wc G E)_\chi=\wc{G}_\chi E_\chi$ for every $\chi\in\subG_0$, see for example the  proof of   \cite[Lem.~2.4]{typeD1}. Analogously $(\wc M\wh M)_{\psi}=\wc{M}_\psi \wh M_\psi$ for every $\psi\in\subM_0$. Assumption \ref{assvi} implies 
        \[(\wc G E\times \calH)_\chi=\wc G_\chi
    	(GEB)_\chi\] for every $\chi\in\subG_0$, 
    	and	\begin{align}
    	\label{whM_psi}
    	(\wc M\wh M\times \calH)_{\psi}=\wc M_\psi \left (  \wh M B  \right )_\psi 
    	\end{align}
    	for every $\psi\in\subM_0$. (Note that $\wh M/M\cong E$.)
    	
    Applying the considerations of the proof of \cite[2.12]{S12} we obtain an $\wc M \wh M$-equivariant bijection $\Omega:\subG\rightarrow \subM$, such that 
    \[((\wc G E)_{\chi}, G, \chi) \geq_c ((\wc G E)_{M,\Omega(\chi)}, M, \Omega(\chi))  \text{ for every $\chi\in\subG$}.\] 
        The bijection $\Omega$ is uniquely determined by the following properties: 
        $\Omega(\subG_0)=\subM_0$, $\Omega(\Irr(G\mid \wc \chi))=\Irr(M\mid \wc\Omega(\wc \chi))$ for every $\wc\chi\in \wc\subG$, and $\Omega$ is $\wc M$-equivariant. More precisely 
        for $\chi\in\subG_0$, $\wc\chi\in\Irr(\widecheck{G}\mid \chi)$ and $\wc\psi:=\wc\Omega(\wc\chi)$ the character $\Omega(\chi)$ is defined to be the one contained in $\Irr( M\mid \wc \psi)\cap \subM_0$. 
        Now note that the equivariance of $\wc \Omega$ implies $(\wc M \wh M)_{\chi}=(\wc M \wh M)_{\Omega(\chi)}$ by the construction for $\chi\in \subG_0$. The remaining values of $\Omega$ are chosen such that $\Omega$ is $\wc{M}$-equivariant. Since $\subM_0$ and $\subG_0$ are $\wh M$-stable, $\Omega$ is then $\wh M\wc M$-equivariant.

        For the proof of the statement it is sufficient to show 
        \[  ((\widecheck G E)_{\chi}, G, \chi) \geq_c ((\widecheck{M}\wh M)_{\psi}, M, \psi) 
        \text{ for every $\chi\in\subG_0$},\] 
        where $\psi=\Omega(\chi)$. Hence we have to find projective representations $\calP$ and $\calP'$ of $(\widecheck G E)_{\chi}$ associated to $\chi$ and $(\wc M \wh M)_\psi$ associated to $\psi$, such that:
        \begin{itemize}
        \item the factor set $\al'$ of $\calP'$ is just the restriction of the factor set $\al$ of $\calP$, and
        \item for every $c\in \Cent_{(\widecheck G E)_{\chi}}(G)$ the matrices $\calP(c)$ and $\calP'(c)$ are scalar matrices to the same scalar.
        \end{itemize}

    A projective representation $\calP$ of $(\widecheck{G} E)_\chi$ associated to $\chi$ can be obtained the following way: 
    let $\wc \chi_0\in\Irr(\wc G_\chi\mid \chi)$ and $\mathcal D$ a representation of $G$ affording $\chi$. We obtain a projective representation $\calP$ of $(\wc G E)_\chi= \wc G_\chi E_\chi$
    associated to $\chi$ by 
    	$$\calP(gd)=\calP_1(g)\calP_2(d) \forevery g\in\wc G_\chi \und d\in GE_\chi,$$ 
    	where $\calP_1$ and $\calP_2$ are representations of $\wc G_\chi$ and $(GE)_\chi$ affording $\wc \chi_0$ and $\Phi_{glo}(\chi)$, respectively, extending $\mathcal D$. 
        
        Analogously, let $\calP'$ be the projective representation of $(\wh M\wc M)_\psi$ associated to $\psi$ defined using $\Phi_{loc}(\psi)$ and $\wc \psi_0\in \Irr(\wc M_\psi\mid \psi)$ where $\wc \psi_0^{\wc M}=\wc \Omega(\wc \chi_0^{\wc G})$. 
        
        By the proof of \cite[Thm.~2.12]{S12},
        the factor sets of $\calP$ and $\calP'$ are determined by the action of $\wh M_\psi$ on $\Irr(\wc G\mid \chi)$ and $\Irr(\wc M\mid \psi)$. Additionally for $z\in \Cent_{\wc G E}(G)=\Cent_E(G)\Z(\wc G)$ with $z=ez_0$ for $e\in \Cent_E(G)$ and $z_0\in \Z(\wc G)$ we see that $\calP(z)$ and $\calP'(z)$ are scalar matrices to $\wh \nu(e) \nu(z_0)$ where $\wh \nu\in \Irr(\Cent_E(G)\mid \Phi_{loc}(\chi))=\Irr(\Cent_E(G)\mid \Phi_{glo}(\psi))$ and $ \nu\in \Irr(Z(\wG)\mid \wc\chi_0)=\Irr(Z(\wG)\mid \wc\psi_0)$ and $z_0\in \Z(\wc G)$ with $z\in z_0 \Cent_E(G)$. This uses that $\Irr(\Cent_E(G)\mid \Phi_{loc}(\chi))=\Irr(\Cent_E(G)\mid \Phi_{glo}(\psi))$. Hence
        the pair $(\calP,\calP')$ has all required properties to give
    	\begin{align}
    	\label{geqcct}
    	((\wc G E)_{\chi}, G, \chi)& \geq_c ((\wc G E)_{M,\psi}, M, \Omega(\chi)).\qedhere 
    	\end{align}
         \end{proof}

\begin{cor}\label{crit_iGalMcK_full}
    Assume in the situation of \Cref{crit_iGalMcK} that the following additionally holds: \medskip 
    
    \begin{asslist}
    \item[(vii)] \label{12vib} 
    The group $B$ normalizes $\wh M$ and  
    \[ \restr[\Phi_{glo}(\chi), \al]|{\wh M_\chi}=[\Phi_{loc}(\psi), \al] \text{ for every }\al \in 	\left (  \wh M  B \right )_\psi \und 
 \chi\in\subG_0,\] 
 where  $\psi\in\subM_0$ with $\widecheck\Omega(\Irr(\widecheck G \mid \chi))=
    \widecheck\Omega(\Irr(\widecheck M \mid \psi))$.
    \end{asslist}
    
    \noindent Then the bijection $\Omega:\subG\lra\subM$ from \Cref{crit_iGalMcK} can be chosen to satisfy additionally 
    \begin{align}
    \label{geqcHct}
    ((\widecheck G E)_{\chi^{\calH}}, G, \chi)_\cH& \geq_c ((\widecheck{M}\wh M)_{\psi^\cH}, M, \psi)_{\cH}  \text{ for every }\chi\in  \subG \und \psi:=\Omega(\chi).
    \end{align}
\end{cor}

In particular, Theorem \ref{crit_iGalMcK} and Corollary \ref{crit_iGalMcK_full} yield Theorem \ref{thm:critintro} from the introduction.

\begin{proof}[Proof of Corollary \ref{crit_iGalMcK_full}] 
As before, it suffices to prove the relation \eqref{geqcHct} for $\chi\in\subG_0$. Let $\psi:=\Omega(\chi)$. We continue using the notation from before, in particular we continue considering the projective representations $\calP$ and $\calP'$ defined in the proof of \Cref{crit_iGalMcK}, associated to $\chi$ and $\psi$.

Next, we compare the relations.
Since the character triples associated with $\chi$ and $\psi$ already satisfy $\geq_c$, we have to ensure two additional equations to verify that the $\cH$ triples also satisfy $\geq_c$. First, we have to see
$(\wh M \wt M\times \cH)_{\chi}=(\wh M \wt M\times \cH)_{\psi}$. This follows from the $\wh M \wt M\times \cH$-equivariance of $\Omega$. Second, we have to see that a pair $(\calP, \calP')$ giving $\geq_c$ additionally satisfies
\[ [\calP,\al](h) = [\calP',\al](h)\]
for every $\al\in (\wc M\wh M B)_\psi$ and $h\in (\wh M \wt M)_\psi$. As $\chi\in \subG_0$ and $\subG_0$ is an $\wh M B$-stable $\wc M$-transversal, 
$(\wh M \wc M B)_\psi=(\wh M B)_\psi \wc M_\psi$ and it is sufficient to compute $[\calP,\al] = [\calP',\al]$ in the case where $\al\in(\wh M B)_\psi$. 

To ensure \cite[Def.~2.4(iv)]{NSV20}, we examine $\calP^{\al}$ and $(\calP')^{\al}$ for $\al\in (\wh M B)_\psi$. Hence, we compute $[\calP,\al]$ and $[\calP',\al]$. The projective representation $\calP$ was defined using the representation $\calP_1$ affording $\Phi_{glo}(\chi)$ and $\calP_2$ affording $\wc \chi_0$. By our assumption $B$ stabilizes $\wh M$ and hence $\calP_1^\al$ and $\calP_2^\al$ are well-defined. 
The representation $\calP_1^{\al}$ affords by definition $\Phi_{glo}(\chi)^{\al}=\Phi_{glo}(\chi) [\Phi_{glo}(\chi),\al]$. 
    	On the other hand $\calP_2^{\al}$ affords by definition $\wc \chi_0^{\al} =\wc\chi_0[\wc \chi_0,\al]$.
        This leads to $[\calP,\al](gd)=[\wc \chi_0,\al](g)[\Phi_{glo}(\chi),\al](d)$ for every $g\in\wc G_\chi$ and $d\in GE_\chi$.

    For computing $[\calP',\al]$, recall $\psi=\Omega(\chi)$ and $\calP'$ is the projective representation of $(\wh M\wc M)_\psi$ associated to $\psi$ defined using $\Phi_{loc}(\psi)$ and $\wc \psi_0\in \Irr(\wc M_\psi\mid \psi)$ with $\wc \psi_0^{\wc M}=\wc \Omega(\wc \chi_0^{\wc G})$. 

    As $\al$ stabilizes $\wh M_\psi$, we can consider $(\restr \calP'|{\wh M_\psi})^{\al}$. By the construction of $\calP'$, it is a representation affording $\Phi_{loc}(\psi)^{\al}=\Phi_{loc}(\psi)[\Phi_{loc}(\psi),\al]$, while   $(\calP'_{\wc M_\psi})^{\al}$ is a representation affording $\wc \psi_0^{\al} =\wc\psi_0 [\wc \psi_0,\al]$.
    As before, this leads to $[\calP',\al](gd)=[\wc \psi_0,\al](g)[\Phi_{loc}(\psi),\al](d)$ for every $g\in\wc M_\psi$ and $d\in \wh M_\psi$. 

    In case $\chi_0^{\wc G}$ is not an extension of $\chi$, the characters $\mu\in \Irr(\wc G/G)$ with $(\chi_0^{\wc G})^\al=\chi_0^{\wc G} \mu$ form a $\Lin(\wc G/\wc G_{\chi_0})$-coset in $\Irr(\wc G/G)$.  By abuse of notation, we write $[\chi_0^{\wc G},\al]$ for this set as well. For $g\in\wc G_{\chi_0}$, we observe that $[\wc \Omega(\chi_0^{\wc G}),\al](g)$ is nevertheless well-defined as well. 
    As $\wc \Omega$ is $\Irr(\wc G/G)\rtimes\spa{\al}$-equivariant by \ref{crit_bij}, we see $[\wc \Omega(\chi_0^{\wc G}),\al]= [\chi_0^{\wc G},\al]$. 
    Taking into account the definition of induced characters we see that $[\wc \psi_0^{\wc M},\al]$ is the set of extensions of $[\wc \psi_0,\al]$, and analogously $[\chi_0^{\wc G},\al]$ are the  extensions of $[\chi_0,\al]$. 
   Altogether we have
    \[ [\wc \psi_0,\al] (g)= [\wc \psi_0^{\wc M},\al](g)= [\wc \Omega(\chi_0^{\wc G}),\al](g)=     [\chi_0^{\wc G},\al](g)=[\chi_0,\al](g) . \] 
    Together with the assumption $[\Phi_{loc}(\psi),\al]=[\Phi_{glo}(\chi),\al]$ we obtain $[\calP,\al]=[\calP',\al] $. This proves the condition of \cite[Def.~2.4(iv)]{NSV20} for $\chi$ and $\psi$.
\end{proof}

\subsection{Changing the overgroup}

In the context of Theorem \ref{crit_iGalMcK} and Corollary \ref{crit_iGalMcK_full}, it will sometimes be useful to change between groups $\wc G$ containing $G$ and satisfying conditions (i)--(iii). For this, we introduce the following, sometimes called a ``butterfly" theorem.
Recall that for a central product $\wc X.Z$ with abelian group $Z$, the irreducible characters of $\wc X.Z$ are of the form $\chi.\mu$ with $\chi\in\Irr(\wc X)$ and $\mu\in \Irr(Z)$ with $\Irr(\wc X\cap Z\mid\chi)=\Irr(\wc X\cap Z\mid \mu)$, see also \cite[\S 12]{IMN}.

\begin{theorem}\label{thm:notbutterfly}
Let $\overline X$ be a finite group with normal subgroups $X$, $\wt X$, and $\wc X$ such that $X\subseteq \wt X\cap \wc X$ and $\overline X= \wc X Z=\wt X Z$ for some $Z\leq \Z(\ov X)$. 

Let $\ov M$ with $\ov X= X \ov M$ and $Z\Z(X)\leq \ov M$. Let $M:=\ov M\cap X$ and $A\leq \Aut(\ov X)_{\wt X, \wc X, \ov M,Z}\times \gal$ and $\subG\subseteq \Irr(X)$ be $A\ov M$-stable and $\subM\subseteq \Irr(M)$ be $A\ov M$-stable. 
Let $\wt M:=\ov M\cap \wt X$ and $\wc M:=\ov M\cap \wc X$.

Assume there exists some $(\Lin(\wt X/X)\rtimes A)$-equivariant bijection 
\[ \wt \Omega\colon \Irr(\wt X\mid \subG)\lra \Irr(\wt M\mid \subM)\]
with $\Irr(\Z(\wt X)\mid \chi)=\Irr(\Z(\wt X)\mid \wt \Omega(\chi)) $  for every $\chi\in \Irr(\wt X\mid\subG)$. Then there exists a
$(\Lin(\wc X/X)\rtimes A)$-equivariant bijection 
\[\wc \Omega\colon \Irr(\wc X\mid \subG)\lra \Irr(\wc M\mid \subM)\]
with 
$\Irr(\Z(\wc X)\mid \chi)=\Irr(\Z(\wc X)\mid \wc \Omega(\chi))$ 
for every $\chi\in \Irr(\wc X\mid\subG)$.
\end{theorem}
\begin{proof}
We see that $\Z(\wc X) Z=\Z(\ov X)=\Z(\wt X)Z$ and hence we can assume that $Z=\Z(\ov X)$ and hence $Z\cap \wc X=\Z(\wc X)$ and $Z\cap \wt X=\Z(\wt X)$.
Let $\chi\in \Irr(\wt X\mid \subG)$ and $\mu_0\in\Irr(Z\cap \wt X\mid \chi)$. Then \[\Irr(\ov X\mid \chi)=\{ \chi.\mu\mid \mu\in\Irr(Z\mid \mu_0)\}\] and hence $\chi.\mu\mapsto \wt\Omega(\chi).\mu$ defines a bijection between $\Irr(\ov X\mid \chi)$ and $\Irr(\ov M\mid \wt \Omega(\chi))$. 
In this way, we obtain a bijection 
\[ \ov \Omega: \Irr(\ov X\mid \subG)\lra \Irr(\ov M\mid \subM).\]
Since $\wt \Omega$ is $\Lin(\wt X/X)\rtimes A$-equivariant and $Z$ is $A$-stable, the bijection $\ov \Omega$ is also $\Lin(\ov X/X)\rtimes A$-equivariant. 

Since $\ov X= \wc X Z$, we see that every $\wc \chi\in\Irr( \wc X)$ is the restriction of an irreducible character $\ov \chi$ and 
\[ 
\{\ov \chi\in \Irr(\ov X)| \restr \ov \chi|{\wc X}= \wc \chi\}
\]
forms an $\Lin(\ov X/\wc X)$-orbit. 
Hence $\restr \ov \chi|{\wc X}\mapsto \restr \ov \Omega(\ov \chi)| {\wc M}$ defines a bijection $\wc \Omega$. 
By this construction $\wc \Omega$ is $\Lin(\wc X/X)\rtimes A$-equivariant and has the stated property. 
\end{proof}

\section{Constructing an equivariant bijection }\label{sec:equivbij}

The main goal of this section is to prove \Cref{thm:thmA}, which simplifies the search of the map in   Theorem \ref{crit_bij} to a question on local extension properties. The proof of \Cref{thm:thmA}
 will finally be completed  in Section \ref{sec:bijpfthmA}, where we 
revisit the construction of Malle \cite{Ma07}, in combination with the equivariance results from \cite{CS13, CS17A}.

\subsection{$d$-Harish-Chandra theory}
Let $(\bH, F)$  be a finite reductive group. That is,  $\bH$ is a connected reductive group defined in characteristic $p$ and $F$ is a Steinberg morphism $F\colon \bH\rightarrow\bH$.  If $F$ is a Frobenius morphism defining $\bH$ over $\FF_q$, and $d$ is some positive integer prime to $p$, then a \emph{$d$-torus} (or $\Phi_d$-torus) is an $F$-stable torus whose order polynomial is a power of the $d$th cyclotomic polynomial $\Phi_d$. A \emph{$d$-split Levi subgroup} of $\bH$ is then a Levi subgroup that is the centralizer in $\bH$ of a $d$-torus of $\bH$. (See \cite[3.5.1]{GeckMalle}.)  In particular, a \emph{Sylow $d$-torus} is a $d$-torus of maximal possible rank, and a \emph{minimal $d$-split Levi subgroup} is the centralizer of a Sylow $d$-torus, as in \cite[3.5.6]{GeckMalle}. In the case that $F$ is not a Frobenius morphism, a modified definition \cite[3.5.3]{GeckMalle} can instead be used to define $d$-split Levi subgroups. 

For an $F$-stable Levi subgroup $\bL$ of $\bG$ we denote by $R_\bL^\bH$ and $\tw{\ast}R_\bL^\bH$ Lusztig's twisted induction and restriction, as in \cite[Def.~3.3.2]{GeckMalle}. (We remark that in the situation we are interested in, these will be independent of the corresponding parabolic subgroup containing $\bL$, so we suppress the notation.) 
A character $\la\in\Irr(\bL^F)$ is called \emph{$d$-cuspidal} if  $\tw{\ast}R_\bM^\bL(\lambda)=0$ for every proper $d$-split Levi subgroup $\bM$ of $\bL$. 

Now, we denote by $\UCh(\bH^F)\subseteq \irr(\bH^F)$ the set of unipotent characters of $\bH^F$. 
A \emph{unipotent $d$-cuspidal pair} is then a pair $(\bL, \lambda)$, where $\bL$ is a $d$-split Levi subgroup of $\bH$ and $\lambda\in\UCh(\bL^F)$ is $d$-cuspidal.
Given a unipotent $d$-cuspidal pair $(\bL, \la)$ for $(\bH, F)$, we let $\UCh(\HF \mid (\bL,\la) )$ denote the unipotent characters in the $d$-Harish-Chandra series of $(\bL, \lambda)$, as defined in \cite[3.5.24]{GeckMalle}. That is, $\chi\in \UCh(\HF \mid (\bL,\la) )$ if and only if $\langle \chi, R_\bL^\bH(\lambda)\rangle\neq 0$. We further let $W_\HF(\bL, \la)$ denote the relative Weyl group of $(\bL, \la)$, defined as the group $\Norm_{\bH^F}(\bL, \la)/\bL^F$.

We note that when $d=1$, the above notions correspond to the usual notions of split Levi subgroups, cuspidal characters, and ordinary Harish--Chandra theory.

\subsection{Galois automorphisms and relative Weyl groups}

First, we work to extend the equivariance results of \cite{CS13, CS17A} to the case of Galois automorphisms in $\galh_\ell$.

	\begin{theorem}\label{thm:Wdfixed}
		Let $(\bH,F)$ be a simple adjoint group, $F$ a Frobenius endomorphism  (not Steinberg)  and $\bS$ a Sylow $d$-torus of $(\bH,F)$ for some positive integer $d$ prime to $p$. Let $W_d$ denote the relative Weyl group $\NNN_\HF(\bS)/\Cent_\HF(\bS)$. Then every $\chi\in \Irr(W_d)$ is $\cH_{[d]}$-invariant.

\end{theorem}

Before fully proving Theorem \ref{thm:Wdfixed}, we begin with the case of exceptional groups. Here, we will see that the conductor plays a useful role.

\begin{notation}
As usual, given a finite group $X$ and a character $\chi\in\irr(X)$, we let $\QQ(\chi)$ denote the field of values for $\chi$.	The \textbf{conductor}  $c(\chi)$ for $\chi$ is then the smallest integer $c$ such that $\QQ(\chi)\subseteq \QQ(e^{2\pi i/c})$.	
\end{notation}

	\begin{lem}\label{lem:exceptconductor}
	Keep the notation of Theorem \ref{thm:Wdfixed}.	If $\bH$ is of exceptional type or if $\HF=\tw 3 \tD_4(q)$, then for any $\chi\in\irr(W_d)$, the conductor $c(\chi)$ divides $\lcm(2,d)$. 
	\end{lem}

	\begin{proof}
According to \cite[Appendix 1]{BrMi}, the group $W_d$ is often a Coxeter group, in which case all characters of $W_d$ are  rational-valued by \cite[5.3.8]{geckpfeiffer}. Hence, it is sufficient to consider the remaining cases.

If $\Cent_\HF(\bS)$ is a torus, the groups $W_d$ can be seen from \cite[Appendix 1]{BrMi}. Using calculations in MAGMA, we see that in each of these cases, $c(\chi)$ divides $\lcm(2,d)$ for all $\chi\in\irr(W_d)$. 
If $\Cent_\HF(\bS)$ is not a torus, then the groups $W_d$ can be seen from  \cite[Table 3]{S09} and from there and some computer calculation as above we see again the statement. 
\end{proof}

\begin{proof}[Proof of Theorem \ref{thm:Wdfixed}]
Let $d':=\lcm(2,d)$.	First, we observe that characters of a cyclic group of order $d'$ are $\cH_{[d]}$-invariant (see Lemma \ref{lem:numtheory}).  
	It is well known that characters of the symmetric group $\Sym_a$ are rational. From this, we see that the characters of the wreath product $\Cy_{d''}\wr \Sym_a$ are $\cH_{[d]}$-invariant for $d''$ dividing $d'$, by the construction of the characters of those groups given for example in \cite[Thm.~4.3.34]{jameskerber}. 
	
Now, if $\Cent_\HF(\bS)$ is a torus, the description of the groups $W_d$ are found in \cite[Appendix~1]{BrMi}. (See also \cite[Ex.~3.5.14, 3.5.15]{GeckMalle}.)
	In this case, the group $W_d$ coincides with the centralizer of a $d$-regular element $w\phi$ of the Weyl group $W$ of $\bH$. The tables list the isomorphism types of $\Cent_W(w\phi)$. 
	According to loc. cit, the group $W_d$ is a wreath product $\Cy_{d''}\wr \Sym_a$ with $d''\mid d'$ as discussed above, whenever $\bH$ is of type $\tA$, $\tB$ or $\tC$. 
	(Note that the authors denote by  $B^{(d'')}_a$ the group $\Cy_{d''}
\wr \Sym_a$.) By the above argument, we see the statement holds in these cases.

Whenever $\bH$ is of type $\tD$ and $W_d$ is not  a wreath product of the structure as mentioned above, we have $W_d$ is a subgroup of $\Cy_{d''}\wr \Sym_a $ of index $2$ given by the elements $(c_1,...,c_a)\si$, where each $c_i\in \Cy_{d''}$, $\si\in\Sym_a$, and  $\prod c_i=1$. In those cases, $d''$ divides  $d'=\lcm(2,d) $. 
Mimicking the construction of characters of wreath products, one can parametrize the characters of such a group as well and sees that its irreducible characters are $\cH_{[d]}$-invariant. 

We must now consider the case where $\Cent_\HF(\bS)$ is not a torus. In the following, we write
$W_d(\bH_1^F)$ for the quotient $\NNN_{\bH_1^F}(\bS_1)/\Cent_{\bH_1^F}(\bS_1)$, where
$(\bH_1,F)$ is a reductive group and
$\bS_1$ denotes a Sylow $d$-torus of $(\bH_1,F)$.
From the descriptions in \cite[Ex.~3.5.14, 3.5.15]{GeckMalle}, we see that the given $W_d$ also occurs as $W_d(\bH_1^F)$ for some reductive group $(\bH_1,F)$.
We have seen above that the characters of $W_d(\bH_1^F)$   are $\cH_{[d]}$-invariant and this proves our statement in all cases where $\bH$ is of classical type and $\HF\not \eq \tw 3 \tD_4(q)$. 
 The remaining cases finally follow from Lemma \ref{lem:exceptconductor}.
\end{proof}

\begin{cor}\label{cor:relweylfixed}
Let $(\bL,\la)$ be a unipotent $d$-cuspidal  pair of a finite reductive group $(\bH,F)$, where $F$ is a Frobenius endomorphism and $\bH$ is simple. Then every $\chi\in\Irr(W_\HF(\bL, \lambda))$ is $\galh_{[d]}$-invariant.
\end{cor}
\begin{proof}
From Theorem \ref{thm:Wdfixed}, the characters of $W_d$ are $\cH_{[d]}$-fixed. Then if $W_d$ is cyclic or $W_\HF(\bL,\la)=W_d(\bH_1^F)$ for some $\bH_1$, we know that  the characters in $\Irr(W_\HF(\bL,\la))$ are $\cH_{[d]}$-invariant. This then covers the cases that $\bL$ is a torus or $\bH$ is of classical type. (The latter can be seen from the description of $\bL$ in \cite[Ex.~3.5.14, 3.5.15]{GeckMalle} and the discussion in \cite[Ex.~3.5.29]{GeckMalle}; see also \cite[pp.~48--51]{BMM}.) If instead $\bH$ is of exceptional type, $W_d$ is noncyclic, and $\bL$ is not a torus, then the groups $W_\HF(\bL, \lambda)$ not covered in Lemma \ref{lem:exceptconductor} are found in \cite[Table 1]{BMM}, or by the Ennola duality \cite[Thm.~3.3]{BMM}, and we see that the conductors still divide $\lcm(2,d)$, using the same type of computer calculations as in Lemma \ref{lem:exceptconductor}.
\end{proof}

    \subsection{$d$-Harish-Chandra Theory and Unipotent Characters}\label{sec:Ibijection}
We next study certain unipotent characters of a finite reductive group $(\bH, F)$, which in our situation will be strongly related to those of degree prime to $\ell$.

\begin{defi}
Let $d$ be a positive integer. We write	$\UCh_{[\Phi_d]'}(\HF)$ for the subset of $\UCh(\HF)$ comprised of those unipotent characters whose generic polynomial, defined as in \cite[(1.31)]{BMM}, is not divisible by $\Phi_d$. 
	
	Let $(\bL,\la)$ be a unipotent $d$-cuspidal  pair of $(\bH,F)$ with $\la\in\UCh_{[\Phi_d]'}(\bL^F)$. Then let
	$\Irr_{[d]'}(W_\HF(\bL,\la))$ be the set of characters $\eta\in \Irr(W_\HF(\bL,\la))$, such that the polynomial degree $\operatorname{D}_\eta$ associated to $\eta$ as in \cite[Thm.~4.2]{Ma07} 
	is not divisible by $\Phi_d$.
	
	 Let $\UCh_{[\Phi_d]'}(\HF \mid (\bL,\la) )=\UCh(\HF \mid (\bL,\la) )\cap \UCh_{[\Phi_d]'}(\HF).$
	\end{defi}

In \cite[Thm.~3.2]{BMM}, it is shown that there is a bijection 
$\mathrm I_{(\bL,\la)}^{(\bH)}: \UCh(\HF \mid (\bL,\la) )\lra \Irr(W_\HF(\bL,\la)).$
By \cite[Thm.~3.4]{CS13}, this bijection is further $\Aut(\HF)_{(\bL, \lambda)}$-equivariant. We next extend this to consider the action of $\galh_\ell$. We let  $E(\bH^F)$ denote a suitable group of graph-field automorphisms for $\bH^F$, as constructed in \cite[2.C]{CS24}.

\begin{prop}\label{prop:HCgalequiv}
	Keep the situation of Theorem \ref{thm:Wdfixed}, where now $d:=d_\ell(q)$ for some odd prime $\ell\neq p$.	Let $(\bL,\la)$ be 
a unipotent $d$-cuspidal pair such that $\bL:=\Cent_{\bH}(\bS)$ is the centralizer of a Sylow $\Phi_d$-torus. Then the bijection $\mathrm I_{(\bL,\la)}^{(\bH)}$ above restricts to an $(\HF\rtimes E(\HF))_{(\bL,\la)}\times \galh_\ell$-equivariant bijection
	  \[ \mathrm I_{(\bL,\la)}^{(\bH)}: \UCh_{[\Phi_d]'}(\HF \mid (\bL,\la) )\lra \Irr_{[d]'}(W_\HF(\bL,\la)).\]
In fact, $\galh_\ell$ acts trivially on both sets, as well as on $\lambda$.
\end{prop}

\begin{proof}
Note that by the degree formula in \cite[Thm.~4.2]{Ma07}, the bijection $\mathrm I_{(\bL,\la)}^{(\bH)}$ from \cite{BMM} indeed induces a bijection on the given subsets. As noted above, it is $(\HF\rtimes E(\HF))_{(\bL,\la)}$-equivariant by \cite[Thm.~3.4]{CS13}.
From Corollary \ref{cor:relweylfixed}, we see that each character of $\irr(W_\HF(\bL, \lambda))$ is $\galh_{[d]}$-invariant, and by Lemma \ref{H}, recall that $\cH_\ell\subseteq \cH_{[d]}$.  It then suffices to show that each member of $\UCh_{[\Phi_d']}(\bH^F)$ is $\galh_\ell$-invariant. (Note that when applied to the components of the adjoint version of $\bL$, this will also yield $\lambda$ is $\galh_\ell$-invariant.)

Let $\chi$ be a unipotent character of $\bH^F$ lying in a Harish-Chandra series $\cE(\bH^F, (\bM, \psi))$ with $\psi$ a cuspidal unipotent character of $\bM^F$, where $\bM$ is an $F$-stable Levi subgroup of $\bH$. From \cite[Prop.~5.6]{geck03}, we have $\QQ(\chi)=\QQ(\psi)$, with two exceptions $\phi_{512, 11}, \phi_{512,12}$ in $\tE_7(q)$ and four exceptions from the $\phi_{4096,\ast}$ family in $\tE_8(q)$.  (Here we use the notation for unipotent characters as in \cite[Ch.~13]{Carter}.) In these exceptional cases, we instead have $\QQ(\chi)=\QQ(\sqrt{q})$. Now, a character $\chi$ occurs in $\UCh_{[\Phi_d]'}(\HF)$ only for integers $d$ that do not appear in the degree polynomials. We remark that, from the list of degree polynomials given in \cite[Sec.~13.9]{Carter}, we see that the degrees of these exceptions are divisible by $\Phi_m$ for each $m$ such that $m$ is even and $\Phi_m$ divides the order polynomial of $\HF$. That is, these exceptions occur in $\UCh_{[\Phi_d]'}(\HF)$ only for odd $d$. Hence by Lemma \ref{sqrt}, we see that these characters are fixed by $\galh_\ell$.

Further,  \cite[Table 1]{geck03} details which cuspidal unipotent characters in $\UCh_{[\Phi_d]'}(\HF)$ have non-rational values.  We summarize these in Table \ref{tab:irrationalcuspidal}. (There, we use $\zeta_n$ to denote an $n$th root of unity and we use the notation of \cite{Carter} for the unipotent characters.)

In all of these cases aside from those of $\tE_7(q)$, we see that the characters are fixed by $\galh_{[d]}$ by definition. In the case of the cuspidal characters of $\tE_7(q)$ with field of values $\QQ(\sqrt{-q})$ when $d\in\{2, 6,10,14,18\}$,  we have $\sqrt{-q}$ is fixed by $\galh_\ell$  by Lemma \ref{2msqrt}. Hence we are left to consider the case $\bM\neq \bH$.

Now suppose $\psi$ is a cuspidal unipotent character as in Table \ref{tab:irrationalcuspidal} of some proper split Levi subgroup $\bM^F$ of $\bH^F$ and such that there is some $\chi\in\cE(\bH^F, (\bM, \psi))\cap \UCh_{[\Phi_d']}(\bH^F)$.  Then we see that $\psi$ is one of the $\tE_6$ or $\tE_7$ characters in the table, and $\bH$ is of type $\tE_7$ or $\tE_8$ in the first case and of type $\tE_8$ in the latter. Further, in the cases that $\psi$ is  $E_6[\theta]$ or $E_6[\theta^2]$, we see from the list of character degrees in \cite[Sec.~13.9]{Carter} that we have $\chi$ is in $\UCh_{[\Phi_d']}(\bH^F)$ only for values of $d$ divisible by $3$. Hence $\QQ(\chi)=\QQ(\psi)=\QQ(\zeta_3)$ is still in the fixed field for $\galh_\ell$. Similarly, in the case $\psi$ is $E_7[\xi]$ or $E_7[-\xi]$, $\chi\in\UCh_{[\Phi_d']}(\bH^F)$ only for values of $d$ of the form $2m$ with $m$ odd, and hence $\QQ(\chi)=\QQ(\psi)=\QQ(\sqrt{-q})$ is in the fixed field of $\galh_\ell$ by Lemma \ref{2msqrt}, completing the proof.
\end{proof}

\begin{center}
\begin{table}[h]
\centering
\begin{tabular}{|c|c|c|c|}
\hline
$H=\bH^F$ & $d$ & $\chi\in\UCh_{[\Phi_d]'}(H)$ cuspidal & $\QQ(\chi)$  \\
\hline 
 $\tG_2(q)$ & $3, 6$ & $G_2[\theta]$, $G_2[\theta^2]$ & $\QQ(\zeta_3)$\\
 \hline
 $\tF_4(q)$ &  $4, 8, 12$ & $F_4[i], F_4[-i]$ & $\QQ(\zeta_4)$\\
 &   $3, 6, 12$ & $F_4[\theta], F_4[\theta^2]$ & $\QQ(\zeta_3)$ \\
 \hline
 $\tE_6(q)$ & $3, 6, 9, 12$ &  $E_6[\theta]$, $E_6[\theta^2]$ & $\QQ(\zeta_3)$ \\
 \hline
 $\tw{2}\tE_6(q)$ & $3, 6, 18, 12$ &  $\tw{2}E_6[\theta]$, $\tw{2}E_6[\theta^2]$ & $\QQ(\zeta_3)$ \\
 \hline
$\tE_7(q)$ & $2, 6, 10, 14, 18$ & $E_7[\xi], E_7[-\xi]$ &  $\QQ(\sqrt{-q})$ \\
 \hline
 $\tE_8(q)$ & $3, 6, 12, 15, 18$ & $E_8[\theta]$, $E_8[\theta^2]$ & $\QQ(\zeta_3)$ \\
 & $6, 18, 24, 30$ & $E_8[-\theta], E_8[-\theta^2]$ & $\QQ(\zeta_3)$\\
 & $4, 8, 12, 20, 24$ & $E_8[i], E_8[-i]$ & $\QQ(\zeta_4)$\\
 & $5, 10, 15, 20, 30$ & $E_8[\zeta], E_8[\zeta^2], E_8[\zeta^3], E_8[\zeta^4]$ & $\QQ(\zeta_5)$\\
 \hline
 
\end{tabular}\caption{Irrational fields of character values for cuspidal characters in $\UCh_{[\Phi_d]'}(H)$}\label{tab:irrationalcuspidal}
\end{table}
\end{center}

\begin{corollary}\label{cor:HCgalequiv}
Let $(\bH, F)$ be a finite reductive group with $F$ a Frobenius morphism defining an $\FF_q$-structure on $\bH$. Let $\ell\neq p$ be an odd prime and write $d:=d_\ell(q)$. Let $\bS$ be a Sylow $d$-torus of $(\bH, F)$, $\bL:=\Cent_\bH(\bS)$ a minimal $d$-split Levi subgroup, and $(\bL, \lambda)$ 
a unipotent $d$-cuspidal pair. Then there is an $ (\HF\rtimes E(\HF))_{(\bL,\la)}\times \galh_\ell$-equivariant bijection 
\[ \mathrm I_{(\bL,\la)}^{(\bH)}: \UCh_{[\Phi_d]'}(\HF \mid (\bL,\la) )\lra \Irr_{[d]'}(W_\HF(\bL,\la)).\] 
In fact, $\galh_\ell$ acts trivially on $\lambda$ and both sets. 
\end{corollary}
\begin{proof}
By arguing exactly as in the first two paragraphs of the proof of \cite[Thm.~3.4]{CS13}, we may assume that $\bG$ is simple of adjoint type, and apply Proposition \ref{prop:HCgalequiv}.
\end{proof}

\subsection{Constructing the bijection $\wt \Omega$}\label{sec:bijpfthmA}

We discuss here the bijection analogous to that of \cite[Sec.~6]{CS17A}, which generalizes that from \cite[Sec.~4]{CS13}, which in turn builds upon that from \cite[Sec.~7]{Ma07}.

Throughout the remainder of the section, we let $\bG$ be a connected reductive group of simply connected type, $F\colon \bG\rightarrow\bG$ a Frobenius morphism such that $(\bG, F)$ is defined over $\FF_q$, let $\ell$ be an odd prime not dividing $q$, and write $d:=d_\ell(q)$.  
Let $\bG\hookrightarrow \wt\bG$ be a regular embedding, as in \cite[Sec.~15.1]{CE04} with $F$ extended as a Frobenius endomorphism of $\wt\bG$. Denote by $\pi^*\colon \wt\bG^*\to\bG^*$ the epimorphism dual to the regular embedding. Note that the diagonal automorphisms of $G=\bG^F$ are induced by conjugation by $\wt G:=\wt\bG^F$. We also denote by $E(\bG^F)$ the group of graph and field automorphisms as, for instance, introduced in \cite[Notation 1.2]{typeD1}.

Given a semisimple element $s\in\bG^{\ast F}$, we let $\cE(\bG^F, s)$ denote the rational Lusztig series corresponding to $s$, and similar for $\wt\bG$.

Let $\bS$ be a Sylow $d$-torus of $(\bG, F)$ and write $\wt \bN:=\Norm_{\wt\bG}(\bS)$, $\wt N:=\wt{\bN}^F$, $\bC:=\Cent_{\wt\bG}(\bS)$, and $C:=\bC^F$. Let $\bC^\ast\leq \wt\bG^\ast$ in duality with $\bC$, so that there is a Sylow $d$-torus $\bS^\ast$ of $(\bG^\ast, F)$ such that $\bC^\ast:=\Cent_{\wt \bG^\ast}(\bS^\ast)$, using \cite[Lem.~3.3]{Ma07}.

Now, by  \cite[Prop. 6.3]{CS17A} and its proof, we have $\Irr(\wGF\mid\Irrl(\GF))$ is in bijection with the set of $(\wt\bG^*)^F$-conjugacy classes of tuples 
$(\bS^\ast, s, \lambda, \eta)$, where 
\begin{itemize}
\item $\bS^\ast$ is a Sylow $d$-torus of $\bG^\ast$; 
\item $s$ a semisimple element of $\bC^{\ast F}$ such that $|\bG^{\ast F}/\Cent_{\bG^\ast}(\pi(s))^F|_\ell =1$ ; 
\item $\lambda\in\UCh(\Cent_{\bC^\ast}(s)^F)\cap \Irrl(\Cent_{\bC^\ast}(s)^F)$; 
\item  $\eta\in\Irrl(W_{\Cent_{\wt\bG^\ast}(s)}(\Cent_{\bC^\ast}(s))^F_\lambda)$, where $W_X(Y):=\Norm_X(Y)/Y$, and
\item any $x\in \Cent_{\bG^\ast}(\pi(s))^F_\ell$ sends $(\bS^*,\lambda ,\eta)$ to a $\Cent_{\bG^\ast}^\circ (\pi(s))^F$-conjugate.
\end{itemize}
Moreover, $\Irr(\wt N\mid\Irrl(N))$ is in bijection  with the tuples $(s, \la, \eta)$ such that $(\bS^\ast, s, \la, \eta)$ belongs to the above set.

\smallskip

We assume that 
for 
\begin{equation*}
\mathfrak{C}:= \{ \chi\in  \Irr(\Cent_\wGF(\bS))\mid  \Irr(N\mid \restr \chi|{\Cent_\GF(\bS)}) \cap \Irrl(N)\neq \emptyset \}
\end{equation*}
there exists some extension map $\wt \Lambda$ for $\mathfrak{C}$ with respect to $C=\Cent_\wGF(\bS)\lhd \wt{N}$ (which, in the context of Theorem \ref{thm:thmA} will be further assumed to be $\galh\ltimes (\Irr(\wt N/N)\rtimes  (\GF\EGF)_\bS)$-equivariant). 
The corresponding bijections are given by the following, using the maps recalled in the preceding section:
\[(\bS^\ast, s, \lambda, \eta)\mapsto \chi_{s,\psi}^{\wt\bG}\in\Irr(\wGF\mid\Irrl(\GF))\cap\mathcal{E}(\wGF, s)\] 
with 
\[\psi:= \mathrm I_{(\Cent_{\bC^\ast}(s),\la)}^{(\Cent_{\wt\bG^\ast}(s))}(\eta)\in \UCh(\Cent_{\wt\bG^\ast}(s)^F \mid (\Cent_{\bC^\ast}(s),\la))\]
and
\[(\bS^\ast, s, \lambda, \eta)\mapsto \Ind_{\wt{N}_{\chi_{s,\lambda}^\bC}}^{\wt N}(\wt\Lambda(\chi_{s,\lambda}^\bC)\cdot\eta)\in \Irr(\wt N\mid\Irrl(N)).\]
Here we use the notation $\chi_{s,\lambda}^\bC$ and $\chi_{s,\psi}^{\wt\bG}$ to denote the unique Jordan decompositions of \cite{DM90}. In addition, $\wt{N}_{\chi_{s,\lambda}^\bC}$ is the stabiliser in $\wt N$ of the character $\chi_{s,\lambda}^\bC\in\mathcal{E}(C,s)$ and \cite[Cor.~3.3]{CS13} yields that we may identify $W_{\Cent_{\wt\bG^\ast}(s)}(\Cent_{\bC^\ast}(s))^F_\lambda$ with $\wt{N}_{\chi_{s,\lambda}^\bC}/C$.

We are now prepared to prove Theorem \ref{thm:thmA} from the introduction.

\begin{proof}[Proof of \Cref{thm:thmA}]
By \cite[Prop.~4.2]{CS13} and \cite[Prop.~6.3]{CS17A}, the above indeed gives a bijection $\wt{\Omega}$.  
 Further, by \cite[Thm.~4.5]{CS13} and \cite[Thm.~6.1]{CS17A}, this bijection can further be chosen to be $\Irr(\wGF/\GF)\rtimes\EGF$-equivariant under the assumption that $\wt\Lambda$ is $\Irr(\wt{N}/N)\rtimes\EGF$-equivariant, and satisfies the remaining properties, aside from $\cH_\ell$-equivariance.  
 
Hence, it now suffices to show that $\wt{\Omega}$ is $\cH_\ell$-equivariant. Note that 
\begin{equation*}
\UCh(\Cent_{\bC^\ast}(s)^F)\cap \Irrl(\Cent_{\bC^\ast}(s)^F)\subseteq \UCh_{[\Phi_d']}(\Cent_{\bC^\ast}(s)^F).
\end{equation*}
Now, by \cite{SV20}, the unique Jordan decomposition for $C$ and for $\wt\bG^F$ from \cite{DM90} is equivariant with respect to $\gal$, in that $(\chi_{s,\psi}^{\wt\bG})^\sigma
=\chi_{s^\sigma, \psi^\sigma}^{\wt\bG}$ and $(\chi_{s,\lambda}^{\bC})^\sigma=\chi_{s^\sigma, \lambda^\sigma}^{\bC}$ for each $\sigma\in\gal$. Here we define $s^\sigma$ as $s^k$, where for the order $o(s)$ of $s$,  $\sigma$ maps $o(s)$-th roots of unity $\zeta_{o(s)}$ to $\zeta_{o(s)}^k$. Then the $\galh_\ell$-equivariance of $\wt\Omega$ follows from the existence of the $\galh_\ell$-equivariant map $\wt{\Lambda}$, together with \Cref{cor:HCgalequiv}.
\end{proof}

\section{The transversal for $\GF\in \{\SL_n(q), \SU_n(q)\}$}
\label{sec:transversal}
The aim of this section is to prove \Cref{thm:typeAtransversalglobal} below. This statement claims that an $E(\GF)$-stable  $\wGF$-transversal in $\Irr(\GF)$ exists that is additionally stable under a suitable subgroup of $\galh_\ell\times \Aut(G)$. For its construction, we apply Generalized Gelfand-Graev Representations that have been shown to be useful in similar contexts in \cite[\S 4]{CS17A}, \cite{Tay}, and \cite{SFT18}.

\subsection{Properties of generalized Gelfand-Graev characters}\label{sec:GGGR}
Let $\GG_m$ be the multiplicative group of $\FF = \overline{\FF}_p$. If $\bH$ is an algebraic group over $\FF$, then we denote by $\cX(\bH)$ the group of characters $\bH \to \GG_m$ and by $\cX^{\vee}(\bH)$ the set of cocharacters $\GG_m \to \bH$. In the case that $\bH$ is a torus, we write $\langle-,-\rangle : \cX(\bH) \times \cX^{\vee}(\bH) \to \ZZ$ for the natural perfect pairing and similarly for its extension by scalars
\begin{equation*}
\langle -,-\rangle : \cV(\bH) \times \cV^{\vee}(\bH) \to \QQ,
\end{equation*}
where $\cV(\bH) := \QQ\otimes_{\ZZ} \cX(\bH)$ and $\cV^{\vee}(\bH) := \QQ\otimes_{\ZZ} \cX^{\vee}(\bH)$. We identify $\cX(\bH)$ and $\cX^{\vee}(\bH)$ with their canonical respective images in $\cV(\bH)$ and $\cV^{\vee}(\bH)$.

Assume $\bG$ is a connected reductive algebraic group, and let $\bT \leqslant \bB \leqslant \bG$ be a maximal torus and Borel subgroup of $\bG$. The pair $(\bT,\bB)$ determines a set of simple and positive roots $\Delta \subseteq \Phi^+ \subseteq \Phi \subseteq \cX(\bT)$. If $\alpha^{\vee} \in \Phi^{\vee} \subseteq \cX^{\vee}(\bT)$ is the coroot corresponding to $\alpha \in \Phi$, then we have a cocharacter
\begin{equation}\label{eq:principal-cochar}
\lambda_{(\bT,\bB)} = \sum_{\alpha \in \Phi^+} \wc\alpha = 2\sum_{\alpha \in \Delta} \wc\omega_{\alpha} \in \cX^{\vee}(\bT)
\end{equation}
where $\wc\omega_{\alpha} \in \langle \Phi^{\vee}\rangle_{\QQ} \subseteq \cV^{\vee}(\bT)$ is the fundamental dominant coweight determined by $\alpha$. In other words, $\langle \alpha,\wc{\omega}_{\beta}\rangle = \delta_{\alpha,\beta}$ is the Kronecker delta for any $\alpha,\beta \in \Delta$.

The cocharacters $\lambda_{(\bT,\bB)}$ arising in \cref{eq:principal-cochar}
will be called the \emph{principal cocharacters of $\bG$}. As any two pairs
$(\bT,\bB)$ are $\bG$-conjugate, the principal cocharacters form a single orbit
for the action of $\bG$ on $\cX^{\vee}(\bG)$ by conjugation.

\begin{lem}
If $\lambda \in \cX^{\vee}(\bG)$ is a principal cocharacter, then $z_{\bG} :=
\lambda(-1) \in \rZ(\bG)$ is a central element satisfying $(z_{\bG})^2 = 1$. It
is independent of $\lambda$.
\end{lem}

\begin{proof}
Assume $\lambda = \lambda_{(\bT,\bB)}$. For any simple root $\alpha \in \Delta$
we have $\alpha(\lambda(c)) = c^2$. Therefore $z_{\bG}$ is in the kernel of
every root so must be contained in $\rZ(\bG)$. It is independent of $\lambda$
because all principal cocharacters are conjugate.
\end{proof}

\begin{remark}\label{rem:zG-SLn}
If $\bG = \SL_n(\FF)$ and $\bT \leqslant \bB$ are the diagonal torus and upper triangular Borel subgroup, then
\begin{equation*}
\lambda_{(\bT,\bB)}(c) = \mathrm{diag}(c^{n-1},c^{n-3},\dots,c^{-(n-3)},c^{-(n-1)})
\end{equation*}
for any $c \in \GG_m$. We then also have that $z_{\bG} = (-1)^{n-1}\Id_n$, where $\Id_n$ is the identity matrix.
\end{remark}

We say $\bH \leqslant \bG$ is a subsystem subgroup if $\bH$ is a closed
connected reductive subgroup of $\bG$ containing a maximal torus of $\bG$. We
are interested in the subsystem subgroups $\bH \leqslant \bG$ that
satisfy the following property:
\begin{equation}\label{eq:princ-inv}
z_{\bG} \cdot \rZ^{\circ}(\bH) = z_{\bH} \cdot \rZ^{\circ}(\bH).
\end{equation}
To see why, we pick a Frobenius endomorphism $F : \bG \to \bG$ defining $\bG$
over $\FF_q$. We say a cocharacter $\lambda \in \cX^{\vee}(\bG)$ is $F$-fixed if
$F(\lambda(c)) = \lambda(c^q)$ and we denote
by $\cX^{\vee}(\bG)^F \subseteq \cX^{\vee}(\bG)$ the set of $F$-fixed
cocharacters. If the pair $(\bT,\bB)$ is $F$-stable, then $\lambda_{(\bT,\bB)}$
is $F$-fixed because it is fixed by all graph automorphisms. Hence, there are
$F$-stable principal cocharacters.

\begin{prop}\label{GGGR}
Assume $\lambda \in \cX^{\vee}(\bG)^F$ is a principal cocharacter. Let $k \in
(\FF_p)^{\times}$ and $c \in (\FF_{p^2})^{\times}$ be elements such that $k =
c^2$. Then $t = \lambda(c) \in \bG$ satisfies the following properties:
\begin{enumerate}
	\item $\Lang(t) \in \langle z_{\bG}\rangle$,
	\item if $\bH \leqslant \bG$ is an $F$-stable subsystem subgroup satisfying
	\cref{eq:princ-inv} and $p$ is good for $\bH$, then ${}^tu$ is $\bG^F$-conjugate
	to $u^k$ for any regular unipotent element $u \in \bH^F$. 
\end{enumerate}
\end{prop}

\begin{proof}
(a). Pick a pair $(\bT,\bB)$ so that $\lambda = \lambda_{(\bT,\bB)}$. As
$\Lang(t) = \lambda(c^{(q-1)})$ it suffices to show that $c^{2(q-1)} = 1$. If $p
=2$ this is clear because $c = k = 1$, so suppose $p \neq 2$. Then $2 \mid (p-1)$
and we have
\begin{equation*}
c^{2(p-1)} = k^{(p-1)} = 1.
\end{equation*}
Now we simply use that $(p-1) \mid (q-1)$.

(b). Let $\bB \leqslant \bG$ be a Borel subgroup of $\bG$ such that $\bB_{\bH} =
\bB \cap \bH$ is the unique Borel subgroup of $\bH$ containing $u$, see
\cite[3.4.10, 12.2.2]{DiMi2}. As $F(u) = u$ we must have $F(\bB_{\bH}) =
\bB_{\bH}$ so $\bB_{\bH}$ contains an $F$-stable maximal torus $\bT \leqslant
\bB_{\bH}$. Consider the principal cocharacter $\lambda' =
\lambda_{(\bT,\bB_{\bH})}$ of $\bH$. This is $F$-stable and by \cite[5.4]{SFT23}
there exists an element $g \in \bB_{\bH}$ such that ${}^gu = u^k$ and
\begin{equation*}
\Lang(g)\cdot \rC_{\bH}^{\circ}(u) = \Lang(\lambda'(c))\cdot \rC_{\bH}^{\circ}(u).
\end{equation*}
As \cref{eq:princ-inv} holds we must have
\begin{equation*}
\Lang(\lambda(c)) \cdot \rZ^{\circ}(\bH) = \Lang(\lambda'(c))\cdot\rZ^{\circ}(\bH).
\end{equation*}
As $\rZ^{\circ}(\bH) \leqslant \rC_{\bG}^{\circ}(u)$ we have $\Lang(g)$ and
$\Lang(t)$ have the same image in $\rC_{\bG}(u)/\rC_{\bG}^{\circ}(u)$ so ${}^tu$
is rationally conjugate to $u^k$.
\end{proof}

We now want to determine which subsystem subgroups satisfy \cref{eq:princ-inv}.
The following exercise with root data shows that Levi subgroups satisfy this
property.

\begin{lem}\label{lem:principal-cochar}
Assume $\bL \leqslant \bG$ is a Levi subgroup of $\bG$ and let $\pi : \bL \to \bL/\rZ^{\circ}(\bL)$ be the canonical quotient map and let $\pi_* : \cX^{\vee}(\bL) \to \cX^{\vee}(\bL/\rZ^{\circ}(\bL))$ be the natural map $\gamma \mapsto \pi\circ\gamma$. If $\lambda \in \cX^{\vee}(\bL) \subseteq \cX^{\vee}(\bG)$ is a principal cocharacter of $\bG$, then the following hold:
\begin{enumerate}
	\item $\pi_*(\lambda) \in \cX^{\vee}(\bL/\rZ^{\circ}(\bL))$ is a principal cocharacter of $\bL/\rZ^{\circ}(\bL)$,
	\item $\pi_*(\lambda) = \pi_*(\lambda')$ with $\lambda' \in \cX^{\vee}(\bL)$ a principal cocharacter of $\bL$.
\end{enumerate}
In particular, $\bL$ satisfies \cref{eq:princ-inv}.
\end{lem}

\begin{proof}
We can assume $\bL$ is a Levi complement of $\bP \leqslant \bG$. We pick a pair $(\bT,\bB)$ of $\bG$ such that $\bT \leqslant \bL$ and $\bT \leqslant \bB \leqslant \bP$. If $\Sigma \subseteq \Phi$ is the root system of $\bL$ then $I = \Sigma\cap\Delta$ is a simple system of roots, where $\Delta \subseteq \Phi$ are defined as above. Recall that we have a natural short exact sequence
\begin{equation*}
0 \longrightarrow \cX^{\vee}(\rZ^{\circ}(\bL))
\longrightarrow \cX^{\vee}(\bT) 
\overset{\pi_*}{\longrightarrow} \cX^{\vee}(\bT/\rZ^{\circ}(\bL)) 
\longrightarrow 0
\end{equation*}
and tensoring with $\QQ$ gives an exact sequence
\begin{equation*}
0 \longrightarrow \cV^{\vee}(\rZ^{\circ}(\bL))
\longrightarrow \cV^{\vee}(\bT) 
\overset{\pi_*}{\longrightarrow} \cV^{\vee}(\bT/\rZ^{\circ}(\bL)) 
\longrightarrow 0
\end{equation*}

The vector space $V := \cV^{\vee}(\bT)$ has a direct sum decomposition $V_I \oplus V_I^{\perp}$, where $V_I = \langle \Sigma^{\vee}\rangle_{\QQ}$ is spanned by the coroots of $\bL$ and
\begin{equation*}
V_I^{\perp} = \{v \in V \mid \langle \alpha,v\rangle = 0\text{ for all }\alpha \in I\}.
\end{equation*}
As $\cX^{\vee}(\rZ^{\circ}(\bL)) = \cX^{\vee}(\bT) \cap V_I^{\perp}$ we may identify $V_I$ with $\cV^{\vee}(\bT/\rZ^{\circ}(\bL))$ and $\pi_* : V \to \cV^{\vee}(\bT/\rZ^{\circ}(\bL))$ with the projection to $V_I$. Now if $\alpha,\beta \in I$ then
\begin{equation*}
\delta_{\alpha,\beta} = \langle \alpha, \wc{\omega}_{\beta}\rangle = \langle \alpha, \pi_*(\wc\omega_{\beta})\rangle.
\end{equation*}
This shows that $\pi_*(\wc\omega_{\alpha})$ is the fundamental dominant coweight of the root system $\Sigma$ with respect to $I$. In particular, if $\bB_{\bL} = \bB \cap \bL$ then by definition
\begin{equation*}
\pi_*(\lambda_{(\bT,\bB)}) = \lambda_{(\pi(\bT),\pi(\bB_{\bL}))} = \pi_*(\lambda_{(\bT,\bB_{\bL})}).
\end{equation*}
Now we simply evaluate this at $-1$.
\end{proof}

The following is also useful to note when dealing with groups of type $\tA$.

\begin{lem}\label{lem:reg-in-Levi}
If $u \in \bG^F$ is regular in a Levi subgroup of $\bG$, then $u$ is regular in an $F$-stable Levi subgroup of $\bG$.
\end{lem}

\begin{proof}
Let $\bL \leqslant \bG$ be a Levi subgroup that contains $u$ as a regular
element. We claim $\rZ^{\circ}(\bL) \leqslant \rC_{\bG}(u)$ is a maximal torus
of the centralizer. Certainly there exists a maximal torus $\bS$ of
$\rC^\circ_{\bG}(u)$ containing the connected center $\rZ^{\circ}(\bL)$. We have
$\bM := \rC_{\bG}(\bS)$ is a Levi subgroup of $\bG$ containing $u$ and
$\rZ^{\circ}(\bL) \leqslant \bS \leqslant \rZ^{\circ}(\bM)$. As $\bL =
\rC_{\bG}(\rZ^{\circ}(\bL))$ we have $\bM \leqslant \bL$ so $\rZ(\bL) \leqslant
\rZ(\bM)$. Now $\rC_{\bM}(u) \leqslant \rC_{\bL}(u)$ implies $\rZ(\bM) \leqslant
\rZ(\bL)$ by \cite[12.2.3]{DiMi2} so $\rZ(\bM) = \rZ(\bL)$. Therefore, $\bS =
\rZ^{\circ}(\bL)$.

By the Lang--Steinberg Theorem, there exists an $F$-stable maximal torus $\bS_1
\leqslant \rC_{\bG}^{\circ}(u)$ so $\rC_{\bG}(\bS_1)$ is an $F$-stable Levi
subgroup containing $u$. The above argument shows that $\rZ^{\circ}(\bL)$ and
$\bS_1$ are both maximal tori of $\rC_{\bG}^{\circ}(u)$ so $\rC_{\bG}(\bS_1) =
{}^g\bL = {}^g\rC_{\bG}(\rZ^{\circ}(\bL))$ for some $g \in \rC_{\bG}(u)$, so $u$
is regular in $\rC_{\bG}(\bS_1)$ because it is regular in $\bL$.
\end{proof}

It is not true in general that a subsystem subgroup $\bH \leqslant \bG$ will
satisfy \cref{eq:princ-inv}. However, this is true if all the quasi-simple components
of $\bG$ are of type $\tA$, because $\bH$ must be a Levi subgroup in this
case. The following shows that this property is also enjoyed by symplectic
groups.

\begin{prop}\label{prop:Sp-involution}
If $\bG = \Sp_{2n}(\FF)$ is a symplectic group, then every subsystem subgroup
$\bH \leqslant \bG$ satisfies \cref{eq:princ-inv}.
\end{prop}

\begin{proof}
We prove this by induction on $n \geqslant 0$. Suppose first that $\bH$ is
contained in a proper subgroup
\begin{equation*}
\bH = \bH_1 \times \bH_2 \leqslant \Sp_{2a}(\FF) \times \Sp_{2b}(\FF) \leqslant \Sp_{2n}(\FF),
\end{equation*}
with $n = a+b$, obtained as the stabilizer of an orthogonal decomposition of
non-degenerate subspace of the natural module for $\bG$. We have $\lambda' =
\lambda_1\times \lambda_2$ with $\lambda_i \in \wc{X}(\bH_i)$ a principal
cocharacter of $\bH_i$. The statement is now clear by induction because
$\lambda(-1) = -\Id_{2n} = (-\Id_{2a}) \times (-\Id_{2b})$. 

So we need only consider where $\bH$ is contained in no proper such subgroup. If
$\bH < \bG$ is proper then $\bH$ must be a maximal Levi subgroup isomorphic to
$\GL_n(\FF)$. The statement in this case follows from
\Cref{lem:principal-cochar}.
\end{proof}

In the following, we denote by $\Gamma_u$ the Generalized Gelfand--Graev
Representation of $\bG^F$ defined by the unipotent element $u \in \bG^F$ as in
\cite{Kaw85}. For any Galois automorphism $\sigma \in \calG$, there exists a
unique $k(\sigma) \in (\FF_p)^{\times}$ such that $\sigma(\zeta) =
\zeta^{k(\sigma)}$ for all primitive $p$th roots of unity $\zeta \in \QQ^{\ab}$.
When combined with the previous statements, we can use this to describe the
action of $\sigma$ on $\Gamma_u$.

\begin{theorem}\label{thm:GalactGGGR}
Assume $(\bG,F)$ is a finite reductive group, with $p$ a good prime for $\bG$,
such that all the quasi-simple components of $\bG$ are of type $\tA$ or
simply connected of type $\tC$. Fix a principal cocharacter $\lambda \in
\cX^{\vee}(\bG)^F$. If $\sigma \in\calG$ and $c \in (\FF_{p^2})^{\times}$ is
such that $k(\sigma) = c^2$, then $t_{\sigma} := \lambda(c)$ satisfies 
\begin{equation*}
\Gamma_u^\si= {}^{t_{\sigma}}\Gamma_u 
\end{equation*}
for all unipotent elements $u \in \bG^F$.
\end{theorem}

\begin{proof}
Write $[\bG,\bG] = \bG_{\tA}\bG_{\tC}$ as a product such that all
the quasi-simple components of $\bG_{\tA}$ are of type $\tA$ and
similarly for $\bG_{\tC}$ (a group of semisimple rank $1$ is of type
$\tA$). Assume $u = u_{\tA}u_{\tC} \in \bG^F$ is a
unipotent element, which is necessarily contained in $[\bG,\bG]$. Both
$\bG_{\tA}$ and $\bG_{\tC}$ are $F$-stable and we claim that
$u_{\tA}$ and $u_{\tC}$ are as well. As $F(u) = u$ we must
have
\begin{equation*}
u_{\tA}^{-1}F(u_{\tA}) 
= 
u_{\tC}^{-1}F(u_{\tC}) 
\in
\bG_{\tA} \cap \bG_{\tC}
\leqslant
\rZ(\bG)
\end{equation*}
so $F(u_{\tA}) \in u_{\tA}\rZ(\bG)$. The statement now follows
from the Jordan decomposition.

It is known that $u_{\tA}$ is regular in a Levi subgroup of
$\bG_{\tA}$ and so there exists an $F$-stable Levi subgroup
$\bH_{\tA} \leqslant \bG_{\tA}$ which contains $u_{\tA}$ as
a regular element by \Cref{lem:reg-in-Levi}. We leave it to the reader to deduce
from \cite[9.3]{SFT23}, by passing to $F$-orbits of quasisimple components, that
there exists an $F$-stable subsystem subgroup $\bH_{\tC} \leqslant
\bG_{\tC}$ containing $u_{\tC}$ as a regular element.

Therefore $\bH = \bH_{\tA}\bH_{\tC}\rZ^{\circ}(\bG)$ is an
$F$-stable subsystem subgroup containing $u$ as a regular unipotent element and
we can see from the construction that $p$ is good for $\bH$ because it is good
for $\bG$. By \cite[Prop.~4.10]{SFT18} we have $\Gamma_u^\si=\Gamma_{u^k}$,
where $k = k(\sigma)$. On the other hand ${}^{t_{\sigma}}\Gamma_u =
\Gamma_{{}^{t_{\sigma}}u} = \Gamma_{u^k}$ by \cite[11.10]{Tay}, because
${}^{t_{\sigma}}u$ and $u^k$ are $\bG^F$-conjugate by \Cref{GGGR}.
\end{proof}

\subsection{The global transversal for type $\tA$}\label{sec:globtransversal}

As in \cite[Notations 2.14, 2.15]{CS24}, we note that $G\lhd \Lang^{-1}(\rZ(\bG))$ and that the diagonal automorphisms are alternatively induced by the action of $\Lang^{-1}(\rZ(\bG))$. The corresponding outer automorphisms of $G$ are indexed by the group $\wt{\rZ}(G):=\rZ(\bG)/[\rZ(\bG), F]$. Indeed, the latter is isomorphic to $\wt{\bG}^F/G\rZ(\wt{\bG})^F=(\bG\rZ(\wt{\bG}))^F/G\rZ(\wt{\bG})^F$ via $gz\mapsto \Lang(g)[\rZ(\bG), F]$ for $g\in \bG$ and $z\in \rZ(\wt\bG)$ with $gz\in\wt\bG^F$. 

It will sometimes be useful in what follows to realize the diagonal automorphisms in these different ways. In particular, while $\wt G:=\wt\bG^F$ is often the natural choice to consider for $\wc G$ in the setting of Theorem \ref{crit_iGalMcK}, we will see here and in the next section that in fact translating to $\Lang^{-1}(\rZ(\bG))$ will yield a more compatible overgroup in the case of type $\tA$, for which we will use Theorem \ref{thm:notbutterfly}.

For this subsection, we now let $\bG$ be simply connected of type $\tA$, so that $\GF\in \{\SL_n(q), \SU_n(q)\}$. 
By Proposition \ref{GGGR}, combined with Theorem \ref{thm:GalactGGGR}, we see that for each $\sigma\in\gal$, there is some $\hat t_\sigma\in \wt{\rZ}(G)$ with $\hat t_\sigma^2=1$
and $t_\sigma\in \Lang^{-1}(\Z(\bG))$  with $\Lang(t_\sigma)[\rZ(\bG), F]=\hat t_\sigma$ such that ${\Gamma_{u}}^{ t_\sigma}=\Gamma_u^\sigma$ for every Generalized Gelfand--Graev Representation $\Gamma_u$ of $G$. 
Note that the $t_\sigma$ described in Proposition \ref{GGGR} is determined by $\sigma$ up to multiplication with the scalar matrix $(-1)^{n-1}\cdot I_{n}$. (We remark, however, that multiplying by any inner automorphism yields another element with the  properties stated.)

Define $\gal_0\leq \Lang^{-1}(\rZ(\bG))\times \gal$ to be the group $\gal_0:=\langle t_\sigma \sigma^{-1} \mid \sigma \in \gal \rangle$ and $\galh_0\leq\Lang^{-1}(\rZ(\bG))\times\galh_\ell$ to be the group $\galh_0:=\gal_0\cap(\Lang^{-1}(\rZ(\bG))\times \galh_\ell)$.  

Now, in \cite[Thm.~4.6]{CS17A}, it is shown that corresponding to each unipotent
conjugacy class $\cC$ of $\bG$, there is associated an $E(\GF)$-stable Generalized Gelfand--Graev Representation of
$G$, which we will denote $\Gamma_{\cC}$. Here we assume $E(\GF)$ is defined
with respect to the pair $(\bT,\bB)$ consisting of the diagonal maximal torus
and upper triangular Borel. In particular, the above discussion applies to
$\Gamma_u=\Gamma_{\cC}$, and we immediately obtain the following:

\begin{cor}\label{cor:GGRs_cG}
 Let $\cC$ be a unipotent conjugacy class of $\bG$  and $\Gamma_\cC$ be the $\E(\GF)$-stable Generalized Gelfand--Graev Representation from \cite[Thm. 4.6]{CS17A}. Then $\Gamma_\cC$ is $\gal_0$-stable.  
\end{cor}

From this, we obtain the desired global transversal, adapted from that in \cite{CS17A}:
\begin{theorem}\label{thm:typeAtransversalglobal}
	The $\EGF$-stable $\wGF$-transversal  
    in $\Irr(\GF)$ from \cite[Thm.~4.1]{CS17A} is also stable under $\gal_0$. 
\end{theorem}

\begin{proof}
If the statement holds, then for every $\wt \chi\in\Irr(\wGF)$, there exists some $\chi\in\Irr(\GF\mid \wt\chi)$ contained in this transversal with
\[
(\wt G E(\GF)\calG_0)_\chi=\wt G_\chi (E(\GF)\calG_0)_\chi,
\]
that also extends to $\GF E(\GF)_\chi$. 
On the other hand, one can construct a transversal as required once it is established that for every $\wt \chi\in\Irr(\wGF)$, some $\chi\in\Irr(\GF\mid \wt\chi)$   satisfies

\[
(\wt G E(\GF)\calG_0)_\chi=\wt G_\chi (E(\GF)\gal_0)_\chi\]
and extends to $\GF E(\GF)_\chi$. 
(Characters with such a property then form a $\calG_0 E(\GF)$-stable set and within this set one can choose a $\calG_0 E(\GF)$-stable transversal, as  every $\wGF$-orbit in $\Irr(\GF)$ has a non-trivial intersection with this set.)

For the following fix some $\wt \chi\in\Irr(\wGF)$ and let us find some $\chi\in\Irr(\GF\mid \wt\chi)$ that satisfies
\[
(\wt G E(\GF)\calG_0)_\chi=\wt G_\chi (E(\GF)\calG_0)_\chi\]
and extends to $\GF E(\GF)_\chi$. 

We mimic the construction of such characters in \cite[Thm 4.1]{CS17A}. 
According to \cite[3.2.18.(iii), 3.2.24.(i)]{Kaw85},  there exists some unipotent class $\cC$ of $\widetilde\bG$ and a Generalized Gelfand--Graev Character $\wt\Gamma$ of $\wt{G}$ associated to $\cC$ such that $\widetilde\Gamma$ satisfies $\langle\widetilde\chi,\widetilde\Gamma\rangle=1$.  Let $\Gamma_\cC$ be the $\EGF \gal_0$-stable Generalized Gelfand--Graev Character associated to $\cC$ from Corollary \ref{cor:GGRs_cG}. 
Then we have  $\widetilde\Gamma=\mathrm{Ind}_{G}^{\wGF}(\Gamma_\cC)$ from the construction in \cite[Thm. 4.6]{CS17A}. 

It follows that there is a unique $\chi \in\Irr(\wGF|\widetilde\chi)$ with multiplicity one in $\Gamma_\cC$ as in \cite[Prop.~4.5]{CS17A}. 
Since $\Gamma_{\cC}$ is $E(\GF)\gal_0$-stable, the character $\chi$ satisfies $(\wc G E(\GF)\calG_0)_\chi=\wc G_\chi (E(\GF)\calG_0)_\chi$. 
As we have chosen the character $\chi$ as in the proof of \cite[Thm.~4.1]{CS17A}, it is clear that $\chi$ also extends to $\GF E(\GF)_\chi$ as $\Gamma_\cC$ extends to $\GF E(\GF)$ (see e.g. \cite[Prop.~4.5(b)]{CS17A}). 
\end{proof}

Denote by $E(\Lang^{-1}(\rZ(\bG)))$ the group generated by the restrictions to $\Lang^{-1}(\rZ(\bG))$ of graph automorphisms and some Frobenius endomorphism as in \cite[Notation 2.1]{typeD1}. In particular, restriction defines a natural surjective homorphism $E(\Lang^{-1}(\rZ(\bG))) \to E(\bG^F)$ whose kernel is generated by $F$.

\begin{rem}\label{rem:butterflytransversal}
We remark that, thanks to Theorem \ref{thm:typeAtransversalglobal}, we also obtain a $\Lang^{-1}(\rZ(\bG))$-transversal $\subG_0$ in $\irr(G)$ that is stable under $E(\Lang^{-1}(\rZ(\bG)))$.
\end{rem}

In what follows, for an integer $k$, we will let  $\left(\frac{k}{q}\right)$ denote the Jacobi symbol when $q$ is odd and define $\left(\frac{k}{q}\right):=1$ when $q$ is a power of $2$. Recall that $(\bT,\bB)$ denotes the diagonal maximal torus and upper triangular Borel subgroup of $\bG$.

\begin{corollary}\label{cor:Bglobal}
Let $G\in\{\SL_n(q), \SU_n(q)\}$ and set $\widecheck{G}:=\Lang^{-1}(\Z(\bG))$, $E:=E(\wc G)$, and $\lambda = \lambda_{(\bT,\bB)}$.
Assume $\sigma \in \mathcal{G}$ is a Galois automorphism and $t_{\sigma} =
\lambda(c)$, where $c \in \mathbb{F}_{p^2}^\times$ satisfies
$c^2=k(\sigma)$. Then the following hold:
\begin{enumerate}
\item $\Lang(t_\sigma)=\left(\frac{k(\sigma)}{ q}\right)^{n-1} \Id_n$;
\item ${}^{t_\sigma} E \subset \langle -\Id_n \rangle E$; and
\item there exists a subgroup 
$B \leq \Cent_{\widecheck{G}}(\mathbf{S}) \times \galh_{\ell}$ 
satisfying 
\[\NNN_{\widecheck{G} E}(Q) B=\NNN_{\widecheck{G} E}(Q) \times \galh_{\ell}\]
and stabilizing $\subG_0$, where $\subG_0$ is the transversal discussed in
Remark \ref{rem:butterflytransversal}. Moreover, the group $B$ normalizes
$\widehat{M}:=\N_{GE}(M)$, where $M=\norm{G}{\bS}$ for $\bS$ a $d_\ell(q)$-Sylow torus of $(\bG, F)$.
\end{enumerate}
\end{corollary}

\begin{proof}
We have $E = \langle \gamma, F_p\rangle$ as in \cite[\S3.2]{CS17A}. Note that $t
:= t_{\sigma}$ is fixed by $\gamma$ and satisfies $t^{-1}F_p(t) =
\lambda(c^{p-1})$ so $\Lang(t_{\sigma}) = \lambda(c^{q-1})$. If $p = 2$ then $t$
is $F_p$-fixed so (a) and (b) are clear in this case. Now assume $p \neq 2$.
Observe that $k:=k(\sigma)$ satisfies $k^{(p-1)/2}=\left(\frac k p\right)$ so
$c^{q-1} = k^{(q-1)/2}=\left(\frac k q\right)$. We get (a) and (b) from the
explicit description of $\lambda$ given in the remark of \Cref{sec:GGGR}.

Now consider (c). Conjugating with $t$ induces an isomorphism $({}^{t} \mathbf{S})^F \cong
\mathbf{S}^F$, as $\Lang(t) \in \Z(\bG)$. In particular, both $\mathbf{S}$ and
${}^t \mathbf{S}$ are Sylow $\Phi_d$-tori and thus $\bG^F$-conjugate. Hence,
there exists some $g' \in \bG^F$ such that ${}^{g' t} \mathbf{S}=\mathbf{S}$. As
$\Norm_{\widecheck{G}}(\mathbf{S})$ is self-normalizing, by conjugacy of Sylow
$\Phi_d$-tori, we get that $g' t \in \Norm_{\widecheck{G}}(\mathbf{S})$. In
particular, there is some $n \in \Norm_{G}(\mathbf{S})$ such that $g_\sigma
tk\in \Cent_{\widecheck{G}}(\mathbf{S})$ where $g_\sigma=n g'$. Moreover,
$g_\sigma t$ normalizes $M=\N_G(\mathbf{S})$. As ${}^{g_\sigma t} (\bG^F
E)=\bG^F E$ it thus follows that $g_\sigma t$ also normalizes
$\widehat{M}=\N_{GE}(M)$. We then define $B:=\langle g_\sigma t_{\sigma}
\sigma^{-1} \mid \sigma \in \mathcal{H} \rangle.$
\end{proof}

Note that the construction of $B$ depends on various choices, and different choices can lead to different actions on characters $\hat{\chi} \in \Irr(\widehat{M}_\chi)$ extending a character $\chi \in \Irr(M)$. Now take two elements in $g_1,g_2 \in \wc{G} \times \mathcal{H}$ which satisfy $g_1 g_2^{-1} \in G \Z(\bG)$ (i.e. their image in $\mathrm{Out}(G) \times \mathcal{H}$ coincide) and $g_1,g_2 \in \Norm_{\widecheck{G}}(\widehat{M})$. We now want to compute the difference between ${}^{g_1} \hat{\chi}$ and ${}^{g_2} \hat{\chi}$. For this, observe that $g_1=t_1 z_1 \sigma$ and $g_2=t_2 z_2 \sigma$ for $\sigma\in\galh$, $t_1,t_2 \in G$, and $z_1,z_2 \in \Z(\bG)$. Then $t:=g_1 g_2^{-1}=t_1 t_2^{-1} z_1 z_2^{-1}$. The following lemma then shows that ${}^{g_1} \hat{\chi}={}^{g_2} \hat{\chi} \la$, where $\la \in \Irr(\widehat{M}_\chi)$ is the linear character defined by $\la(e):=\nu([z_1,e] [z_2^{-1},e])$ with $\nu \in \Irr(\Z(M) \mid \chi)$. 

\begin{lem}
Keep the notation from Corollary \ref{cor:Bglobal}.
Let $\chi \in \Irr(M)$ and $\nu \in \Irr( \Z(M) \mid \chi)$. 
Let $t=gz \in \ov G=\wc G \Z(\wbG)$ with $g\in  G$ and $z\in \Z(\ov \G)$ and assume that $t$ acts on $\wh M_\chi$ with $\chi^t=\chi$. Then every extension $\hat{\chi}$ of $\chi$ to $\wh{M}_\chi$ satisfies 
$$ [\hat \chi, t](e)=\nu([z,e]) \text{ for every }e \in \wh{M}_\chi.$$
\end{lem}

\begin{proof}
By assumption $t$ normalizes $\wh{M}_\chi$ . In particular, $t$ normalizes $ G E':=G\wh{M}_\chi$ for some $E'\leq E$. As $\wc G$ normalizes $G E'$ this also shows ${}^{z} (GE')=GE'$ and that $z$ normalizes $\wh{M}'=\Norm_{\widecheck{G} E'}(\mathbf{S})$. This shows that $[z,GE']=[z,E'] \subset GE' \cap \Z(\wbG)=\Z(G)$. Since $g \in G$ and normalizes $M$, we have $g \in M$ and so ${}^g \hat{\chi}=\hat{\chi}$. Hence, 
$${}^t \hat{\chi}(e)=^{z}\hat{\chi }(e)=\hat{\chi}([z,e] e)=\nu([z,e]) \hat{\chi}(e)$$ for all $e \in \wh{M}_\chi$. Here the last equation follows from the fact that $[z,e]\in \Z(M)$ and the matrix associated to a representation affording $\hat{\chi}$ at the element $[z,e]$ is the scalar matrix with value $\nu([z,e])$.
\end{proof}

\section{Local results for $\GF\in\{\SL_n(q), \SU_n(q)\}$ }\label{sec:transversallocal}

\subsection{The local transversal for type $\tA$}

Recall that the automorphisms of $\GF$ induced by $\wGF$ are also induced by $\wc G:=\Lang^{-1}(Z(\bG))$. Consequently, the corresponding outer automorphisms of $\GF$ are labeled by elements of $\Z(\bG)/[\Z(\bG),F]$. 

 For every $\sigma\in \cH_\ell$ and $\zeta$ an $\ell'$-root of unity, there exists some integer $k$ such that $\sigma(\zeta)=\zeta^{k}$ and we denote this integer below by $k_{\ell,m}(\sigma)$, where $m$ is of order $\zeta$. Note that $k_{\ell,m}(\sigma)$ is independent of the choice of $\zeta$ and determines a unique element of $\ZZ/m\ZZ$. Given a group $G$ and a prime $\ell$ we write 
$k_{\ell,G}(\sigma)$ for $k_{\ell,|G|_{\ell'}}(\sigma)$.

The following is a result of \cite{SoniaSpaeth}, and produces the desired local transversal in the case of type $\tA$.

\begin{theorem}[Petschick--Sp{\"a}th \cite{SoniaSpaeth}]\label{sonia1}
Let $\ell$ be an odd prime and write $\galh:=\galh_\ell$.
Let $q$ be a prime power with $\ell\nmid q$, $(\bG,F)$ with $\GF\in \{\SL_n(q),\SU_n(q)\}$, $d:=d_\ell(q)$ be the order of $q$ in $(\ZZ/\ell \ZZ)^\times $ and $\bS$ a Sylow $\Phi_d$-torus of $(\bG,F)$. Take $\wGF\in\{\GL_n(q),\GU_n(q)\}$ from a regular embedding of $\bG$.

The groups $\wt N:=\NNN_\wGF(\bS)$, $\wc N:=\NNN_{\wc G}(\bS)$ and $\wh N:=\NNN_{\GF E(\wc G)}(\bS)$ satisfy the following:
\begin{asslist}
		\item  
There exists some $B' \wh N$-stable $\wc N$-transversal in $\Irr(N)$ for 

\[B'=\left \{ (t,\sigma)\in  \wc N \times \cH \mid 
        \Lang(t)\in \left(\frac {k_{\ell,\GF}(\sigma)} q\right)^{n-1} [\Z(\bG),F]\right\}.\]

\item There exists some $\Irr(\wt N/N)\rtimes \galh$-equivariant extension map $\wt \Lambda$ \wrt $\wt L \lhd \wt N$ where  $\wt L:=\Cent_\wGF(\bS)$ and $ \wt N:=\NNN_\wGF(\bS)$.
	\end{asslist}
\end{theorem}

\subsection{From $\wt{G}$ to $\Lang^{-1}(\rZ(\bG))$}
We remark that from the construction in Corollary \ref{cor:Bglobal}, we can see that the transversal $\mathbb{T}_d$ in  \Cref{sonia1}   is also $B$-stable, where $B$ is as in Corollary \ref{cor:Bglobal}. 
Further, when Theorem \ref{thm:thmA} is combined with  \Cref{sonia1}  , we obtain a bijection $\wt\Omega$ as in Theorem \ref{thm:thmA}.    From this, we now obtain the corresponding statements replacing $\wt{G}$ with $\widecheck{G}:=\Lang^{-1}(\rZ(\bG))$, using Theorem \ref{thm:notbutterfly}.

\begin{theorem}\label{thm:notbutterflyA}
Continue to assume $G:=\bG^F\in\{\SL_n(q), \SU_n(q)\}$ and let $B$ be as defined in \Cref{cor:Bglobal}. Let $\widecheck{G}:=\Lang^{-1}(\rZ(\bG))$ and $E:=E(\widecheck{G})$.
Let $M:=\norm{G}{\bS}=N$, 
$\wh{M}:=\norm{GE}{\bS}$, 
and $\widecheck{M}:=\norm{\widecheck{G}}{\bS}$. 
Then there is a $B\wh{M}$-stable $\widecheck{M}$-transversal in $\irr(M)$, call it $\subM_0$, and an
$\Irr(\widecheck M/M)\rtimes(\wh M B )$-equivariant bijection
	$$\widecheck \Omega: \Irr(\widecheck{G}\mid \Irrl(G ))\lra\Irr(\wc{M}\mid \Irrl(M)),$$
	such that 
	$\Irr\left(\Z(\wc G)\mid \wc\Omega(\wc\chi)\right)=\Irr\left({\Z(\wc G)}\mid \wc\chi\right)$ 
    for every $\wc\chi \in \Irr(\widecheck{G}\mid \Irrl(G ))$.
\end{theorem}
\begin{proof}

 Recall that we write $G:=\bG^F$ and $\wt{G}:=\wt{\bG}^F$. We will work with the group $\wt{G}\wc{G}$. Note that $\rZ(\wc{G})=\rZ(\bG)$ and $\rZ(\wt G)\cap\rZ(\wc G)=\rZ(G)$.
 Write $Z:=\Lang^{-1}_{\rZ(\wt\bG)}(\rZ(\bG))\leq \rZ(\wt\bG)$, where 
$\Lang_{\rZ(\wt\bG)}\colon \rZ(\wt \bG)\rightarrow\rZ(\wt\bG)$ is the Lang map viewed as a map on $\rZ(\wt\bG)$, which is surjective since $\rZ(\wt\bG)$ is connected. Then the group $\wt{G}\wc{G}$ can be written as a central product $\wt G\wc G=\wt{G}Z=\wc{G}Z$. 

Let $\wt M:=\norm{\wt{G}}{\bS}$. Thanks to  \Cref{thm:thmA}, combined with  \Cref{sonia1}(ii), we have an $\cH_{\ell}\ltimes (\Irr(\wGF/\GF)\rtimes (\GF\EGF)_{\bS})$-equivariant bijection
	$$\wt \Omega: \Irr(\wt G\mid \Irrl(G ))\lra\Irr(\wt M\mid \Irrl(M )),$$
	which further satisfies that $\Irr\left(\Z(\wt G)\mid \wt\Omega(\wc\chi)\right)=\Irr\left({\Z(\wt G)}\mid \wt\chi\right)$ for every $\wt\chi \in \Irr(\wt G \mid \Irrl(G )).$ 
Note that these groups satisfy the hypotheses in Theorem \ref{thm:notbutterfly}, with $\subG=\Irrl(G)$ and $\subM=\Irrl(M)$, so the existence of $\wc\Omega$ now follows.

  Finally, the existence of the stated transversal $\subM_0$ follows from \Cref{sonia1}(i).
\end{proof}

We are now ready to prove Corollary \ref{cor:equivcondtypeA} from the introduction.

\begin{proof}[Proof of Corollary \ref{cor:equivcondtypeA}]
Combining Theorem \ref{thm:notbutterflyA} with Corollary \ref{cor:Bglobal}, we obtain the first statement. Theorem \ref{crit_iGalMcK} then yields Condition \ref{cond:equivcondition}, as long as $\bS$ can be chosen such that $M=\norm{G}{\bS}$ contains $\norm{G}{P}$ for a Sylow $\ell$-subgroup $P$ of $G$. By \cite[Thm.~5.14]{Ma07}, this holds except for the case $\ell=3=n$ and $q\equiv 4,7\pmod 9$ if $G=\SL_3(q)$ or $q\equiv 2,5\pmod 9$ if $G=\SU_3(q)$. In the latter cases, the full inductive McKay--Navarro conditions have been obtained by Johansson in  \cite[Thm.~A]{johansson}, completing the claim.
\end{proof}

\begin{remark}
In particular,
Corollary \ref{cor:equivcondtypeA} tells us that to complete the inductive conditions in the case that  $G\in\{\SU_n(q), \SL_n(q)\}$, it remains to show the condition in Corollary \ref{crit_iGalMcK_full} that 
$ [\Phi_{glo}(\chi), \al]_{\wh M_\chi}=[\Phi_{loc}(\psi), \al]$ 
for every $\al \in 	\left (  \wh M  B \right )_\psi \und 
 \chi\in\subG_0.$

\end{remark}

\section{On rationality of extensions of unipotent characters}\label{sec:extunip1}

We now change our attention in the next two sections to unipotent characters, with the aim of studying rationality properties of their extensions and proving Theorem \ref{thm:unipcriterion} from the introduction. 

\subsection{Clifford theory and Galois automorphisms}

We will often be in the situation that a Galois automorphism stabilizes a character but not its extension to an overgroup. For this, recall the   notation introduced in Definition \ref{def:extmap}. Namely, if $X \lhd Y$ are finite groups and $\tilde{\rho} \in \Irr(Y)$ is such that $\rho:=\Res_X^Y(\tilde{\rho})$ is irreducible, then for $\sigma \in \gal$ stabilizing $\rho$, we let $[\tilde{\rho},\sigma]$ be the unique linear character in $ \Irr(Y/X)$ such that $\tilde{\rho}^\sigma=\tilde{\rho} \cdot[\tilde{\rho},\sigma]$.

We will also use the following lemma to glue extensions:

\begin{lem}\label{Gallagher}
	Assume that $X \lhd Y$ such that $Y/X$ is abelian and suppose that $X\leq Y_1, Y_2\leq Y$ such that $Y=Y_1 Y_2$ and $X=Y_1 \cap Y_2$. If $\la$ extends to $Y$ then the map $\Irr(Y \mid \la) \to \Irr(Y_1 \mid \la) \times \Irr(Y_2 \mid \la), \phi \mapsto (\Res_{Y_1}^{Y}(\phi), \Res_{Y_2}^{Y}(\phi)),$ is bijective.
\end{lem}

\begin{proof}
	This is a direct consequence of Gallagher's Theorem \cite[Cor.~6.17]{Isa}.
\end{proof}

 The following lemma might be interesting for other character-theoretic questions.

\begin{lem}\label{clifford}
	Assume that $X \lhd Y$ such that $Y/X$ cyclic of order $r=2^a$ or $3^a$ for some integer $a \geq 0$. Suppose that $y \in Y$ is such that $Y= \langle y,X \rangle$. Assume that $\rho\in \Irr(X)$ is a $Y$-stable character that appears with multiplicity coprime to $r$ in a $K_0 X$-module $M$ for some field extension $K_0$ of $\mathbb{Q}_\ell$ and suppose that there exists an isomorphism $\phi: M \to {}^y M$ of $K_0 X$-modules such that $\phi^r=\mu \cdot \mathrm{Id}$ for some $ \mu \in K_0^\times$. Let $\mu_0 \in \overline{\mathbb{Q}}^\times_\ell$ such that $\mu_0^r=\mu^{-1}$. Then the character $\rho$ has an extension $\tilde{\rho} \in \Irr(Y)$ such that $[\tilde{\rho},\sigma](y)=\mu_0^\sigma \mu_0^{-1}$
	 for any $\sigma \in \mathrm{Gal}(\overline{K}_0 / K_0(\rho))$. 	
\end{lem}

\begin{proof}
	We first assume that $r=2^a$.
	Set $K:=K_0(\mu_0)$.
	We extend the $K X$-module $K \otimes_{K_0} M$ to a $KY$-module $\tilde{M}$ by setting $xy^i.m:= \mu_0^i\phi^i(xm)$ for $m \in M, x \in X$ and $i \in \mathbb{Z}$. Let $m_1,\dots,m_n$ be a basis of $M$ as $K_0$-vector space.

	It follows that $m_1,\dots,m_n$ is also a basis of $\tilde{M}$ as a $K$-vector space. Let $\sigma \in \mathrm{Gal}(\overline{K}_0 / K_0(\rho))$ and $\lambda_\sigma \in \Irr(Y/X)$ the unique linear character defined by $\lambda_\sigma(y)=\mu_0^\sigma \mu_0^{-1}$. The action of $g \in Y$ on ${}^\sigma\tilde{M}$ is given by $g \ast m:=\sigma(g.m)$. Therefore, we have $$xy^i \ast m_j=\sigma(\mu_0^i) \phi^i(x m_j)=(\mu_0^i)^{-1} \sigma(\mu_0^i) \mu_0^i \phi_i(xm_j)=\lambda_\sigma(y^i) xy^i. m_j.$$
	It thus follows that ${}^{\sigma} \tilde{M} \cong \tilde{M} \otimes \lambda_\sigma$ as $K Y$-modules.
	
	Let $\tilde{\phi}$ be the character of $\tilde{M}$, $Y_0=X \lhd Y_1 \lhd \dots \lhd Y_a =Y$ the chain of normal subgroups between $X$ and $Y$ such that $Y_i/Y_{i-1}$ is cyclic of prime order. Denote by $\phi_i$ the restriction of $\tilde{\phi}$ to $Y_i$. By induction it is now possible to find unique characters $\rho_i\in \Irr(Y_i)$, $i=0,\dots,a$ such that $\langle \phi_i, \rho_i \rangle$ is odd and $\rho_i$ restricts to $\rho_{i-1}$ for all $i=1,\dots, a$ and $\rho_0=\rho$. We denote $\tilde{\rho}=\rho_a$.  
We now show by induction that $\rho_i^\sigma \Res^{Y}_{Y_i}(\lambda_\sigma)^{-1}=\rho_i$ for $i=0,\dots,a$ and $\sigma \in \mathrm{Gal}(\overline{K}_0 / K_0(\rho))$. The case $i=0$ is by definition and if $\rho_i^\sigma \Res^{Y}_{Y_i}(\lambda_\sigma)^{-1}=\rho_i$ for some $i$, then it follows that  $\rho_{i+1}^\sigma \Res^{Y}_{Y_{i+1}}(\lambda_\sigma)^{-1}$ is also an extension of $\rho_i$ and we have $$\langle \rho_{i+1}^\sigma \Res^{Y}_{Y_{i+1}}(\lambda_\sigma)^{-1}, \phi_{i+1} \rangle= \langle  \rho_{i+1}, \Res^{Y}_{Y_{i+1}}(\lambda_\sigma)^{-1} \phi_{i+1}^\sigma \rangle=\langle  \rho_{i+1}, \phi_{i+1} \rangle$$ is odd. Since these properties uniquely determine $\phi_{i+1}$ it follows that $\rho_{i+1}^\sigma \Res^{Y}_{Y_{i+1}}(\lambda_\sigma)^{-1}=\rho_{i+1}$. Hence, by induction $\tilde{\rho}^\sigma=\tilde{\rho} \lambda_\sigma$. This shows the statement in the case where $r=2^a$.

Now for the case $r=3^a$ one argues in the same way but this time one constructs inductively characters $\rho_i\in \Irr(Y_i \mid \rho_{i-1})$, $i=0,\dots,a$ such that $b_i:=\langle \phi_i, \rho_i \rangle$ is distinct from the multiplicities of all other extensions of $\rho_{i-1}$ to $Y_i$ in $\phi_i$. This again gives a chain of characters which are uniquely determined by the sequence of multiplicities $b_0,\dots,b_a$.
\end{proof}

\subsection{Extensions of unipotent characters}

We assume that $(\bG,F)$ is a finite reductive group, with $F$ a Frobenius endomorphism. Since we are only concerned with unipotent characters, we shall assume throughout that $\bG$ is of adjoint type.
Recall that we can always write $E(\bG^F)=\langle \Gamma,F_0 \rangle$, where $\Gamma$ is the group generated by graph automorphisms and where $F_0$ is a Frobenius endomorphism that acts trivially on $W$, the Weyl group of $\bG$, and $F_0^r=F^\delta$, where $\delta$ is the smallest integer such that $F^\delta$ acts trivially on $W$. In \cite{DudasMalle}, the authors compute the character values of extensions of unipotent characters to $G \langle \gamma \rangle$ for $\gamma\in\Gamma$. In this section, we build on their work and show how one can compute the extensions of unipotent characters to $G E(\bG^F)$. 

As in \cite{DudasMalle}, for $w \in W$, we denote by $X_w$ the associated Deligne--Lusztig variety. Let $R_w$ be the associated Deligne--Lusztig virtual character defined by $$R_w(g):=\sum_{i \in \mathbb{Z}} (-1)^i \mathrm{Tr}(g, H^i_c(X_w,\mathbb{Q}_\ell) )$$ for $g \in G$. For any unipotent
character $\rho$ of $G$, $F^\delta$ acts by the same eigenvalue of Frobenius $\omega_\rho$ on any $\rho$-isotypic component $H^i_c(X_w,\overline{\mathbb{Q}}_\ell)_\rho$ of any $\ell$-adic cohomology group of any $X_w$, up to multiplication
by integral powers of $q^\delta$. (See \cite[4.2.21]{GeckMalle}.)

\begin{proposition}\label{cuspidal value}
	Let $\rho \in \UCh(G)$ be a cuspidal character with Frobenius eigenvalue $\omega_\rho$. Let $\zeta_0$ be an $r$th root of $\omega_\rho$. Then there exists an extension $\tilde{\rho} \in \Irr(G \langle F_0 \rangle )$ such that $\mathbb{Q}_\ell(\tilde{\rho})=\mathbb{Q}_\ell(\zeta_0)$. Moreover, for any $\sigma \in \mathrm{Gal}(\overline{\mathbb{Q}}_\ell / \mathbb{Q}_\ell)$ we have $[\tilde{\rho},\sigma](F_0)=\zeta_0^\sigma \zeta_0^{-1}$.

\end{proposition}

\begin{proof}
Note that by \cite[Proposition 4.5.5]{GeckMalle} and \cite[Table 1]{geck03}, we have $\mathbb{Q}(\rho)=\mathbb{Q}(\omega_\rho)$ and $\mathbb{Q}(\omega_\rho) \in \{\mathbb{Q},\mathbb{Q}(i), \mathbb{Q}(\zeta_3)\}$, unless $\rho$ is either one of the two cuspidal characters in $\tE_7(q)$ with $\mathbb{Q}(\rho)=\mathbb{Q}(\sqrt{-q})$ or one of the four cuspidal characters in $\tE_8(q)$ with $\mathbb{Q}(\rho)=\mathbb{Q}(\zeta_5)$.

We first deal with the non-exceptional cases, i.e. when $\mathbb{Q}(\rho) \in \{\mathbb{Q},\mathbb{Q}(i),\mathbb{Q}(\zeta_3)\}$. For a prime number $t$ we let $F_t$ be a generator of the Sylow $t$-subgroup of $\langle F_0 \rangle$ such that $F_t^{r_t}=F^{\delta}$ for some integer $r_t$.
By Lemma \ref{Gallagher} it suffices to construct the extensions of $\rho$ to $G \langle F_t \rangle$.
By \cite[Cor.~6.6(a)]{Navarro_book} we can in this case assume that $t=2$ or $t=3$.
	By \cite[Thm.~ 2.18]{Lusztig} (and \cite[2.19]{Lusztig} resp. \cite[Proposition 2.5,2.7]{DudasMalle} when $\G$ admits a graph automorphism), there exists some $w \in W$
	such that $\langle R_w, \rho \rangle \in \{\pm 1\}$.
	In particular, there exists some integer $i$ such that $\rho$ has multiplicity coprime to $t$
	in the character of $H^i_c(X_w,\overline{\mathbb{Q}}_\ell)$. Hence, there exists a generalized eigenspace $H^i_c(X_w,\overline{\mathbb{Q}}_\ell)_\mu$ for an eigenvalue $\mu$ for the action of $F^{\delta}$ on $H^i_c(X_w,\overline{\mathbb{Q}}_\ell)$ such that $\rho$ occurs with multiplicity not divisible by $r$ in the character of this eigenspace.  Recall (see e.g. \cite[4.2.21]{GeckMalle}) that the eigenvalue $\mu$ coincides with the Frobenius eigenvalue $\omega_\rho$ up to multiplication by an integral power of $q^\delta$, say $\omega_\rho=\mu q^{\delta a}$. Set $K_0:=\mathbb{Q}_\ell(\mu)=\mathbb{Q}_\ell(\rho)$.
	
	 We denote by $M:=H^i_c(X_w,K_0)_{\mu,\rho}$ the $\rho$-isotypic component of the generalized eigenspace of $H^i_c(X_w,K_0):=H^i_c(X_w,\mathbb{Q}_\ell) \otimes_{\mathbb{Q}_\ell} K_0$. Note that since $F_t$ acts trivially on $W$ we have an action $F_t:X(w) \to X(w)$ on the associated Deligne--Lusztig variety. The associated map $F_t^\ast:M \to {}^{F_t} M$ is an isomorphism of $K_0 G$-modules and satisfies $(F_t^\ast)^{r_t}=F^\delta=\mu\cdot \mathrm{Id}$. Thus, by Lemma \ref{clifford} there exists some extension $\tilde{\rho}$ of $\rho$ such that $K_0(\mu_0)=K_0(\tilde{\rho})$, where $\mu_0 \in \overline{\mathbb{Q}}_\ell^\times$ satisfies $\mu_0^{r_t}=\mu$. Note that $F_t$ is a Frobenius defining an $\mathbb{F}_{p_0}$-structure on $\bG$, for some power $p_0$ of $p$. Then we have $p_0^{r_t}=q^\delta$. Hence, if we define $\zeta_{0,t}:=\mu_0 p_0^a$ then we have $\mathbb{Q}_\ell(\zeta_{0,t})=\mathbb{Q}_\ell(\mu_0)$ and $\zeta_{0,t}^{r_t}=\mu q^{\delta a}=\omega_\rho$. 
	 
	 Finally, we consider the exceptional cuspidal characters $\rho$ of $G=\tE_7(q),\tE_8(q)$ with $\mathbb{Q}(\rho)\in \{\mathbb{Q}(\sqrt{-q}),\mathbb{Q}(\omega_\rho)=\mathbb{Q}(\zeta_5)\}$. In this case, by \cite[Thm.~ 6.1]{LusztigCoxeter} and \cite[Section 7.3]{LusztigCoxeter}, we observe that $\rho$ appears with multiplicity one in the generalized eigenspace of an eigenvalue $\mu$ in the cohomology group of $H^r_c(X(c),\mathbb{Q}_\ell)$, where $c$ is the Coxeter element of $W$ and $r$ is the rank of $\bG$. We extend $H^r_c(X(c),\mathbb{Q}_\ell(\mu))_{\mu}$ to a $\mathbb{Q}_\ell(\zeta_0)(G \langle F_0 \rangle)$-module $\tilde{M}$ as in the proof of Lemma \ref{clifford}. By Frobenius reciprocity, there exists a unique extension $\tilde{\rho} \in \Irr(G \langle F_0 \rangle \mid \rho)$ which appears with multiplicity one in the character of $\tilde{M}$. As in Lemma \ref{clifford}, this yields the result of the proposition.
\end{proof}
 
We finish this section by recalling the corresponding result for the extensions to $G \langle \gamma \rangle$, where $\gamma$ is a graph automorphism. For this recall that the unipotent characters of $G=\tA_n(\varepsilon q)$ are parametrized by partitions $\rho$ of $S_{n+1}$. We write $\chi_\rho$ for the corresponding unipotent character and denote by $\hat{\chi}_\rho$ a fixed extension of $\chi_\rho$ to $G \langle \gamma \rangle$.

\begin{theorem}[Dudas--Malle]\label{values graph}
	Let $\gamma$ be a graph automorphism of a simple algebraic group $\bG$ of simply connected type. Then all unipotent characters have a rational extension to their inertia group in $G \langle \gamma \rangle$ except possibly in the following cases:
	\begin{enumerate}

		\item $G$ is of type $\tA_n(\varepsilon q)$ and $\chi=\chi_\rho \in \mathrm{Unip}(G)$ has $2$-core $\rho_0$ of size $r \equiv 2,3 \mod 4$. Then $\mathbb{Q}(\hat{\chi}_\rho)=\mathbb{Q}(\hat{\chi}_{\rho_0})$.
		\item $G$ is of type $\tE_6(\varepsilon q)$ and $\chi$ lies in the $2$-Harish-Chandra-series (resp. Harish-Chandra series if $\varepsilon=-1$) of a Levi subgroup of type $\tA_5(\varepsilon q) (q+\varepsilon)$. 
	\end{enumerate}
In all these cases, $\mathbb{Q}(\chi)=\mathbb{Q}(\sqrt{\varepsilon q})$.
\end{theorem}

\begin{proof}
	This is a direct consequence of the description given in Dudas--Malle \cite[Thm.~3]{DudasMalle}. Note that if $G=\tA_n(q)$ then this is not stated there but one can immediately check this with the condition given there. For the information about the extensions in $G=\tE_6(\varepsilon q)$ one can consult \cite[Ex.~3.1]{DudasMalle} and the references given there.
\end{proof}

\subsection{Generalized Harish-Chandra theory and extensions of characters}

We suppose now that $(\Levi,\lambda)$ is a unipotent $e$-cuspidal pair of $G$ for some integer $e \geq 1$. Since $\bG$ is simple of adjoint type, the Levi subgroup $\Levi$ has connected center as well and the unipotent character $\lambda$ is trivial on $\Z(\Levi)^F$. Hence, we can consider $\lambda$ as a unipotent character of $\Levi^F/\Z(\Levi)^F=(\Levi/\Z(\Levi))^F$. We denote $\Levi_0:=\Levi/\Z(\Levi)$, a group of adjoint type, and $E(\Levi_0^F)$ the group of automorphisms associated to it. Denote $N:=\N_{G}(\Levi)$  and $\hat{N}:=\N_{G E(\bG^F)}(\Levi)$ and observe that $L_0:=\Levi_0^F \lhd \hat{N}_\la/\C_{\hat{N}}(L_0) \lhd L_0 E(\Levi_0^F)_{\lambda}$. By \cite[Prop.~2.2]{MaExt} the character $\la$ extends to $L_0 E(\Levi_0^F)_\la$. Denote by $E_0 \leq E(\Levi_0^F)$ the maximal subgroup of $E(\Levi_0^F)$ such that $\la$ has an extension $\hat{\la}$ to $E_0$ with $\mathbb{Q}(\hat{\la})=\mathbb{Q}(\la)$.

\begin{lem}\label{extensions local}
Keep the notation and assumptions of above. If $\G$ is of classical type or $e=1$, then we have $N_\la/\C_{N}(L_0) \leq L_0 E_0$ and $\hat{N}_\la/\C_{\hat{N}}(L_0)=L_0 E(L_0)_\la$. 
\end{lem}

\begin{proof}

Assume first that $\bG$ is of classical type, i.e. $\G$ is of type $\mathrm{X}_n$ for a symbol $\mathrm{X} \in \{\tA,\tB,\tC,\tD \}$. Then $\Levi_0$ is of type $\mathrm{X}_m$ for $m \leq n$. In particular, if $\G$ is not of type $\tA_n$, then $\Levi_0$ has no component of type $\tA_{n'}$ unless  $\G$ is of type $\tD_n$ and $\Levi_0$ is of type $\tD_2=\tA_1 \times \tA_1$ or $\tD_3=\tA_3$. We discuss which outer automorphisms $\hat{N}$ induces on $L_0$ when $\bG$ is of classical type. Firstly, the group $\N_G(\Levi)$ induces only inner automorphisms on $L$ unless  $\G$ is of type $\tD$. In the latter case $\N_G(\Levi)$ can also induce a graph automorphism $\gamma_0$ on $L_0$. However, in this case Theorem \ref{values graph} shows that the extensions of $\lambda$ to $L_0 \langle \gamma_0 \rangle$ are rational. Note that this also applies in the case that $\Levi_0$ is of type $D_2=A_1 \times A_1$, as the unipotent characters also have a rational extension to the wreath product. This shows the first statement.
   For the second statement, observe that $F_0$ induces again a Frobenius endomorphism on $\Levi_0$ which acts trivially on the Weyl group associated to a maximal torus of $\Levi_0$, while the graph automorphism of $E(\bG^F)$ induces the graph automorphism of $\Levi_0$.

		Assume finally that $\G$ is of exceptional type and $e=1$. A case-by-case analysis of the possible cases for $(\Levi,\lambda)$ together with the arguments from before shows again the claim. (For instance $\Levi$ has only a component of type $\tA$ when $G=\tE_6(-q)$ and $\Levi^F=(q-1)\tA_5(-q)$; but note that in this case the cuspidal unipotent character has a rational extension to its automorphism group as $\N_G(\Levi)$ induces only inner automorphisms on $\Levi_0^F$). As above, this again yields the statement.
\end{proof}

\begin{theorem}\label{valuesF0}
	Let $\bG$ be a simple adjoint algebraic group and $F$ a Frobenius endomorphism. Assume that $F_0$ is a Frobenius endomorphism with $F_0^r=F^\delta$ for some integer $r$. Then every unipotent character $\chi$ of $G=\bG^F$ extends to a character $\tilde \chi \in \Irr(G \langle F_0 \rangle)$ with $\mathbb{Q}_\ell(\tilde{\chi})=(\mathbb{Q}_\ell(\chi))(\zeta_0)$, where $\zeta_0$ is such that $\zeta_0^r=\omega_{\chi}$. Moreover, for any $\sigma \in \mathrm{Gal}(\overline{\mathbb{Q}}_\ell / \mathbb{Q}_\ell(\chi))$ we have $[\tilde{\chi},\sigma](F_0)=\zeta_0^\sigma \zeta_0^{-1}$.
	
\end{theorem}

\begin{proof}
Suppose that $\chi$ lies in the Harish-Chandra series of $(\Levi_1,\lambda_1)$ and note that $\omega_\chi=\omega_{\lambda_1}$ by \cite[Prop.~4.2.23]{GeckMalle}. Let $W_1:=W_G(\Levi_1,\lambda_1)$ be the corresponding relative Weyl group. Since $F_0^r=F^\delta$, the pair $(\Levi_1,\lambda_1)$ is $F_0$-stable. In addition, it follows that $W_1^{F_0}=W_1$.
	
	By Lemma \ref{extensions local} and Proposition \ref{cuspidal value} there exists an extension $\hat{\lambda}_1$ of $\lambda_1$ to $\N_{G \langle F_0 \rangle}(\Levi_1)_{\lambda_1}$ which satisfies $\mathbb{Q}_\ell(\hat{\lambda}_1)=\mathbb{Q}_\ell(\zeta_0)$. 
    In particular, the assumptions of \cite[Lem.~2.5]{RSF22} are satisfied and the character has an extension to $G \langle F_0 \rangle$ with the desired properties.
\end{proof}

\begin{remark}
	In the case where $G$ is twisted and $F_0^2=F^2$ the automorphisms $\gamma$ and $F_0$ coincide as automorphisms of $G$. In this case, our results give a different proof of the results of Dudas--Malle.
\end{remark}

\section{Galois-compatible extensions for unipotent characters}\label{sec:extunip2}

We start by describing the irrationalities encountered in the previous section in terms of generalized Harish-Chandra theory. 
We will deal especially with groups $\G$ whose root system is of classical type, as the exceptional groups seem to require a tedious case-by-case analysis. However, in the situation we are interested here, the latter will also be discussed in Sections \ref{sec:charvalslocal} and \ref{sec:exceptionalunipotentext}.

 Throughout this section, $\ell$ is an odd prime. Recall that $d$ is the order of $q$ modulo $\ell$.
 Moreover, if $\bG$ is of type $\tA_n(\varepsilon q)$ or $\tD_n(\varepsilon
q)$, we denote by $d'$ the order of $\varepsilon q$ modulo $\ell$. We need the following lemma (note that part (b) is a generalization of Lemma \ref{lem:numtheory}(c),(d)):

\begin{lem}\label{number theory}
 {Let $p_0$ be a power of a prime.
	\begin{enumerate} 

		\item Let $q=p_0^r$ for some integral power $r$ and  $a$. If $a \mid d$, then $X^r-\zeta_{a}$ has a zero in $\mathbb{Q}_\ell$.
		\item Let $q^2=p_0^r$ for some integer $r$. If $X^r- \varepsilon q$ has no zero in $\mathbb{Q}_\ell$, then $d'$ is even. 
	\end{enumerate}	
}\end{lem}

\begin{proof}
	
 Note that $a \mid d$ and $d \mid \ell-1$, so $\ell \nmid a$. We can therefore also assume that $\ell \nmid r$. Moreover, it suffices to treat the case when $a=m^i$ for some prime $m$ with $\ell \nmid m$ and $i > 0$.
    
    Suppose that $X^r-\zeta_{m^i}$ has no zero in $\mathbb{Q}_\ell$. This is equivalent to $\zeta_{m^i}$ having no $r$th-root in $\mathbb{F}_\ell^\times$. This is equivalent to $(\ell-1)_m \leq r_m m^{i-1}$, where $r_m$ is the $m$-part of $r$. Let $d_0$ denote the order of $p_0$ modulo $\ell$. Note that $(d_0)_m \leq (\ell-1)_m \leq r_m m^{i-1}$. Therefore, $d_m=\frac{(d_0)_m}{((d_0)_m,r)} \leq m^{i-1}$. 
	
	For the second part, observe that if $d'$ is odd, then $\varepsilon q=\zeta_{d'}$ in $\mathbb{F}_\ell^\times$ for a primitive $d'$th root of unity and thus $X^r-\varepsilon q$ has a zero in $\mathbb{Q}_\ell$ by the first part.
\end{proof}

\begin{lem}\label{HCodd}
	Let $\bG$ be of classical type $\tB,\tC,\tD$ such that $\bG^F$ is not of type ${}^3 \tD_4(q)$.
	Suppose that $(\Levi,\lambda)$ is a unipotent $d$-cuspidal pair for $d$ odd and assume that $\lambda$ lies in the Harish-Chandra series of $(\Levi_1,\lambda_1)$. Then there exists a unipotent $1$-cuspidal pair $(\Levi_1',\lambda_1')$ of $(\bG,F)$ such that $\mathcal{E}(G,(\Levi,\lambda)) \subset \mathcal{E}(G,(\Levi_1',\lambda_1'))$ and the characters $\lambda_1,\lambda_1'$ restrict to the same unipotent cuspidal character $\pi$ of $[\Levi_1',\Levi_1']^F=[\Levi_1,\Levi_1]^F$.
\end{lem}

\begin{proof}
	Since this is a statement about unipotent characters, we can choose for $(\bG,F)$ the most convenient group. We therefore assume that $(\bG,F)$ is as in \cite[Ex.~3.5.15]{GeckMalle}.
	We have $L:=\Levi^F={H} \times \GL_1(q^d)^{(n-m)/d}$, where $H$ is of the same rational type as $G$, and so $L_1:=\Levi_1^F={H}' \times \GL_1(q)^{m-m'} \times  \GL_1(q^d)^{(n-m)/d}$ for some $H' \leqslant H$ which is again of the same type. It follows that the projection of $\lambda$ onto $H$ is contained in the $1$-Harish-Chandra series of the unique unipotent cuspidal character $\pi$ of $H' \times \GL_1(q)^{m-m'}$. We denote $\Levi_0:=\C_{\bG}(\Z(\Levi)_{\Phi_1})$ the smallest $1$-split Levi subgroup of $\bG$ containg $\Levi$ so that $L_0:=\Levi_0^F= H \times \GL_d(q)^{(n-m)/d}$. By transitivity of Lusztig induction, for every $\chi \in \mathcal{E}(G,(\Levi,\lambda))$ there exists a character $\psi \in \Irr(L_0)$ such that $\langle \psi, R_{\Levi}^{\Levi_0}(\lambda) \rangle \neq 0$ and $\langle \chi, R_{\Levi_0}^{\bG}(\psi) \rangle \neq 0$. Now $(\Levi,\lambda)$ is a $d$-cuspidal pair of $\Levi_0$. 
	
	We claim that there exists a unique $1$-cuspidal pair $(\Levi_1',\lambda_1')$ of $L_0$ such that $\mathcal{E}(L_0,(\Levi,\lambda)) \subset \mathcal{E}(L_0,(\Levi_1',\lambda_1'))$ and $[\Levi_1,\Levi_1]^F=[\Levi_1',\Levi_1']^F$. Indeed, consider the $1$-split Levi subgroup $\Levi_1$ of $\Levi_0$ with $F$-fixed points $L_1':=\Levi_1'^F:=L_1'={H}' \times \GL_1(q)^{n-m'}$. Recall that all unipotent characters $\GL_d(q)$ are contained in the principal unipotent Harish-Chandra series. Hence, the Harish-Chandra series of $(\Levi_1',\lambda_1')$ with $\lambda_1'=\pi \otimes 1$ contains all unipotent characters of $L_0=H \times \GL_d(q)^{(n-m)/d}$ whose projection onto $H$ is contained in the Harish-Chandra series of $(H' \times \GL_1(q)^{n-m'},\pi)$. Since the projection of every $\psi \in \mathcal{E}(L_0,(\Levi,\lambda))$ onto $H$ also lies in the Harish-Chandra series of $(H' \times \GL_1(q)^{n-m'},\pi)$, it follows that $\mathcal{E}(L_0,(\Levi,\lambda)) \subset \mathcal{E}(L_0,(\Levi_1',\lambda_1'))$ as claimed.
	
	Hence, $\psi$ lies in the Harish-Chandra-series of $(\Levi_1',\lambda_1')$ and thus also $\chi$ lies in the Harish-Chandra series of $(\Levi_1',\lambda_1')$.
\end{proof}

\begin{remark}
The previous lemma could also be proved using the language of symbols. For this recall that unipotent characters of $\bG^F$ are in our situation labeled by certain symbols as in \cite[Section 4.4]{GeckMalle}. Moreover, since $d$ is odd, two characters lie in the same $d$-Harish-Chandra series if they have the same $d$-core. On the other hand, it is easy to see that two characters with the same $d$-core have the same $1$-core and thus lie in the same Harish-Chandra series.
\end{remark}

Recall that if $G=\tA_n(\varepsilon q)$, then $d'$ denotes the order of $\varepsilon q$ modulo $\ell$.

\begin{corollary}\label{cor:unipsamecore}
	Assume that $G=\tA_n(\varepsilon q)$ and denote $2_{+}:=2$ and $2_{-}:=1$. Suppose that $(\Levi,\lambda)$ is a $d'$-cuspidal pair for $d'$ even and assume that $\lambda$ lies in the $2_\varepsilon$-Harish-Chandra series of $(\Levi_{2_\varepsilon},\lambda_{2_\varepsilon})$. Then there exists a ${2_\varepsilon}$-cuspidal pair $(\Levi_{2_\varepsilon}',\lambda_{2_\varepsilon}')$ of $(\bG,F)$ such that $\mathcal{E}(G,(\Levi,\lambda)) \subset \mathcal{E}(G,(\Levi_{2_\varepsilon}',\lambda_{2_\varepsilon}'))$ and the characters $\lambda_{2_\varepsilon},\lambda_{2_\varepsilon}'$ restrict to the same unipotent $2_\varepsilon$-cuspidal character $\pi$ of $[\Levi_{2_\varepsilon}',\Levi_{2_\varepsilon}']^F=[\Levi_{2_\varepsilon},\Levi_{2_\varepsilon}]^F$.
\end{corollary}

\begin{proof}
	This can be proved exactly as in Lemma \ref{HCodd}, now using the structure given in \cite[Ex.~3.5.14]{GeckMalle}.
\end{proof}

Since unipotent characters of $\tA_{n}(\varepsilon q)$ are parametrized by characters of $\Irr(S_{n+1})$ and $d$-Harish-Chandra theory for these characters can be encoded purely combinatorially (see \cite[Cor.~4.6.5]{GeckMalle}) we obtain the following equivalent statement:

\begin{corollary}\label{same core}
	Let $\chi, \chi'$ be characters of $S_{n}$. If $\chi$ and $\chi'$ have the same $d$-core for an even integer $d$, then $\chi$ and $\chi'$ have the same $2$-core.
\end{corollary} 

\begin{proof}
	This follows directly from the previous corollary. We give a combinatorial proof of the fact which may be of independent interest. We observe that it is clearly enough to show that removing a $d$-rim hook is equivalent to successively removing $2$-rim hooks. We show this by induction on $d$ (the case $d=2$ being trivial). Now suppose that $\la \in \Irr(S_n)$ with associated $\beta$-set $\beta(\lambda)=\{\beta_1,\beta_2,\dots\}$. Assume that $\la$ has a $d$-hook, i.e. there exists some $i$ such that $\beta_i-d \notin \beta(\lambda)$. Hence, by removing this $d$-hook we get the partition $\lambda_d$ with new $\beta$-set $\beta(\lambda_d)=\beta(\lambda) \setminus \{\beta_i\} \cup \{\beta_i-d\}$. We let $0 \leq d_1 < d$ be the maximal even integer such that $\beta_i-d_1 \in \beta(\lambda)$. By removing $2$-hooks we can thus obtain the partition $\lambda'$ with $\beta(\lambda')=\beta(\lambda) \setminus \{ \beta_i-d_1\} \cup \{ \beta_i-d\}$. If $d_1=0$ then we are done. Otherwise, observe that $\beta(\lambda_d)=\beta(\lambda') \setminus \{\beta_i-\} \cup \{ \beta_i-d_1\}$ can be obtained from $\lambda'$ by removing a $d-d_1$ hook. By induction removing a $d-d_1$ hook can now be seen as successively removing $2$-hooks. Hence, the claim follows.
\end{proof}

In general, $d$-Harish-Chandra theory does not preserve Frobenius eigenvalues, in contrast with the case of ordinary Harish-Chandra theory. (See \cite[Prop.~4.2.23]{GeckMalle}.) However, it turns out that in the situation we consider here, they are indeed preserved:

\begin{corollary}\label{FrobeniusEV}
	Suppose that $\bG$ is of type $\tB,\tC,\tD$ and let $(\Levi,\lambda)$ be a unipotent $d$-cuspidal pair for $d$ odd. Then for every $\chi \in \cE(G, (\Levi,\lambda))$, we have $\omega_\chi=\omega_\lambda$.
\end{corollary}

\begin{proof}
	Since we know that the Frobenius eigenvalue is invariant under Harish-Chandra series by \cite[Prop.~4.2.23]{GeckMalle}, we know by Lemma \ref{HCodd} that $$\omega_\chi=\omega_{\lambda_1'}=\omega_{\pi}=\omega_{\lambda_1}=\omega_\lambda$$ for every $\chi \in \mathcal{E}(G,(\Levi,\lambda))$.
\end{proof}

Let $\hat{G}=G E(\bG^F)$ and $\hat{N}:=\N_{\hat G}(\Levi)$. Denote by $\mathcal{E}(\hat{G}\mid (\Levi,\lambda))$ the set of characters of $\hat{G}$ which cover a character in the $d$-Harish-Chandra series $\mathcal{E}(G, (\Levi,\lambda))$ of $(\Levi,\lambda)$.

\begin{proposition}\label{classical}
	Let $\bG$ be of classical type and suppose that $\chi$ is a unipotent character of $G=\bG^F$ in the $d$-Harish-Chandra series of $(\Levi,\lambda)$. Then for every $\hat{\lambda} \in \Irr(\hat{N}_\la \mid \lambda)$ there exist $\hat{\chi} \in\mathcal{E}( \hat{G}_\chi \mid (\Levi,\lambda)) \cap \Irr(\hat{G}_\chi \mid \chi)$  such that for every $\sigma \in \mathrm{Gal}(\overline{\mathbb{Q}}_\ell / \mathbb{Q}_\ell)$ we have $$\Res_{\hat{N}_\chi}^{\hat{G}_\chi}([\hat{\chi},\sigma])=\Res_{\hat{N}_\chi}^{\hat{N}_\la}([\hat{\la},\sigma]).$$
	  In particular, $\mathbb{Q}_\ell(\hat{\chi})=\mathbb{Q}_\ell(\hat{\lambda})$.
\end{proposition}
\begin{proof}
	By Lemma \ref{extensions local} it suffices to find extensions $\hat{\chi}\in \Irr(\hat{G}_\chi)$  and $\hat{\lambda} \in \Irr(L_0 E(\Levi_0^F)_{\lambda})$ of $\lambda$ with the desired properties. 
	Recall that the group $E(\bG^F)$ is generated by a (possibly trivial) graph automorphism $\gamma$ and a Frobenius endomorphism $F_0$ such that $F_0^r=F^{\delta}$ and $E(\bG^F)_\chi=\langle F_0 \rangle \langle \gamma \rangle_\chi$. Similarly, we can write $E(\Levi_0^F)$ as $E(\Levi_0^F)=\langle \gamma',F_0' \rangle$. Moreover, by \cite[Prop.~2.2]{MaExt}, the character $\chi$ extends to $G E(\bG^F)_\chi$. Thus if $\hat{\chi}_1$ is an extension of $\chi$ to $G \langle F_0 \rangle$ and $\hat{\chi}_2$ is an extension of $\chi$ to $G \langle \gamma \rangle_\chi$, then there exists a unique extension $\hat{\chi} \in \Irr(\hat{G}_\chi \mid \chi)$ extending both $\hat{\chi}_1$ and $\hat{\chi}_2$. In particular, $\mathbb{Q}(\hat{\chi})=\langle \mathbb{Q}(\hat{\chi}_1),\mathbb{Q}(\hat{\chi}_2) \rangle$. We can therefore consider extensions to $\langle F_0 \rangle$ and $\langle \gamma \rangle_\chi$ separately.

	We first consider extensions to $G \langle \gamma \rangle_\chi$. By Theorem \ref{values graph}, if $\bG$ is not of type $\tA$, then all unipotent characters have a rational extension to their inertia group in $G \langle \gamma \rangle$. Hence, we can suppose that $\bG$ is of type $\tA$.
	Recall that $d'$ is the order of $ \varepsilon q$ modulo $\ell$. Therefore, $\mathbb{Q}_\ell(\sqrt{ \varepsilon q})=\mathbb{Q}_\ell(\sqrt{\zeta_{d'}})$. Using Lemma \ref{lem:numtheory}, we have $\mathbb{Q}_\ell(\sqrt{\zeta_{d'}})=\mathbb{Q}_\ell(\zeta_{d'})=\mathbb{Q}_\ell$ if $d'$ is odd and $\mathbb{Q}_\ell(\sqrt{\zeta_{d'}})=\mathbb{Q}_\ell(\zeta_{2d'})$ otherwise. In the case that $d'$ is odd we therefore always have $\mathbb{Q}_\ell(\hat{\chi}_2 )=\mathbb{Q}_\ell=\mathbb{Q}_\ell(\hat{\lambda}_2)$, where $\hat{\lambda}_2 \in \Irr(L_0 \langle \gamma' \rangle \mid \la)$.

	We can therefore assume that $d'$ is even.
	Assume that $\chi=\chi_\rho$ is a unipotent character of $\tA_{n}(\varepsilon q)$ corresponding to the partition $\rho \in \Irr(S_{n+1})$. Then by Lemma \ref{same core} if $d'$ is an even integer, the character $\rho$ and its $d'$-core $\rho_{d'}$ have the same $2$-core. {Note that $\lambda$ is the character corresponding to $\rho_{d'}$, by \cite[Cors.~4.6.5, 4.6.7]{GeckMalle}.} In particular, by Theorem \ref{values graph} we have $\mathbb{Q}(\hat{\chi}_2)=\mathbb{Q}(\hat{\lambda}_2)$ for $\hat{\lambda}_2 \in \Irr(L_0 \langle \gamma' \rangle \mid \la)$.

	Let us now consider extensions to $G \langle F_0 \rangle$. By Lemma \ref{valuesF0} we know that there exists an extension $\hat{\chi}_1$ to $G \langle F_0 \rangle$ satisfying $\mathbb{Q}_\ell(\hat{\chi}_1)=\mathbb{Q}_\ell(\zeta_0)$ with $\zeta_0^r=\omega_{\lambda_1}=\omega_{\chi}$, in the notation of Lemma \ref{valuesF0}.
	
	For the characters of $\tA_{n}(-q)$ we have $\omega_\chi \in \{1,-q\}$. Assume that $\chi_\rho$ is a unipotent character of $\tA_{n}(-q)$ corresponding to the partition $\rho \in \Irr(S_{n+1})$. If $d'$ is odd, then $\mathbb{Q}_\ell(\zeta_0)=\mathbb{Q}_\ell$  by Lemma \ref{number theory}. On the other hand, if $d'$ is even, as argued above, the character $\rho$ and its $d'$-core $\rho_{d'}$ have the same $2$-core. In particular, 
$\mathbb{Q}_\ell(\hat{\chi}_1)=\mathbb{Q}_\ell(\hat{\lambda}_1)$ for a suitable character $\hat{\la}_1 \in \Irr(L_0 \langle F_0' \rangle \mid \la)$ by Lemma \ref{valuesF0}.
	
	Now if $G$ is not of type $\tA_n(-q)$, then we have $\omega_{\lambda} \in \{\pm 1\}$ and so $\mathbb{Q}_\ell(\zeta_0)=\mathbb{Q}_\ell$ whenever $d$ is even by Lemma \ref{number theory}. In the case where $d$ is odd we have $\omega_{\chi}=\omega_\lambda$ by Corollary \ref{FrobeniusEV}, so   $\mathbb{Q}_\ell(\hat{\chi}_1)=\mathbb{Q}_\ell(\hat{\lambda}_1)$ by Lemma \ref{valuesF0}.
\end{proof}

\subsection{Character values of local characters}\label{sec:charvalslocal}

We keep the notation from earlier in this section, though $\bG$ is not necessarily assumed to be of classical type here. Recall that we write $\hat G:=GE(\GF)$ with $E(\GF)=\langle \Gamma, F_0\rangle$ and $\Gamma$ is the subgroup of graph automorphisms and $F_0$ a Frobenius which acts trivial on $W$. Further, recall that $\delta$ is the smallest integer such that $F^\delta$ is trivial on $W$, and that $F^\delta=F_0^r$.

As before, let $(\Levi,\lambda)$ be a unipotent $d$-cuspidal pair for $\bG^F$ and set $W_{\hat{G}}(\Levi,\la):=\hat{N}_\la/L$.

\begin{lem}\label{weyl value}
	Every character $\eta \in \Irr(W_{{G}}(\Levi,\lambda))$ has an extension to $W_{{\hat{G}}}(\Levi,\lambda)_\eta$ with values in $\mathbb{Q}_\ell$.
\end{lem}

\begin{proof}
	Observe that every irreducible character of $W_G(\Levi,\lambda)$ has its values in $\mathbb{Q}_\ell$ by Corollary \ref{cor:relweylfixed}.  

    We first consider the extension of characters of $W_G(\Levi,\la)$ to their inertia group in $W_{G_0}(\Levi,\la)$, where $G_0:=G \Gamma$. 
    
The case where $G=\tD_4(q)$ can be checked with \cite[Table 1]{Ma07}. For this we note that if $\Gamma_0 \leq \Gamma$ and $|\Gamma_0| \neq 3$ it is easily verified with MAGMA that all characters in $W_{G \Gamma_0}(\Levi,\la)$ have character values in $\mathbb{Q}_\ell$. If $|\Gamma_0|=3$ this is not necessarily the case when $d \in \{1,4\}$. However, every character in $W_G(\Levi,\la)$ has a character field contained in $\mathbb{Q}(\zeta_d)$ where $\zeta_d$ is a primitive $d$th root of unity. In particular, every character is $3$-rational and has according to \cite[Corollary 6.6]{Navarro_book} a unique $3$-rational extension to its inertia group in $W_{G \Gamma_0}(\Levi,\la)$, which moreover has the same field of values.

In all the other cases, $W_{{G}}(\Levi,\lambda)$ is always isomorphic to a direct factor of $W_{G_0}(\Levi,\lambda)_\eta$, unless possibly when $G$ is of type $\tD_{2n}(q)$ and $\eta$ is $\Gamma$-stable, see the proof of \cite[Thm.~4.6]{Ma07}.
    In the first case, it is clear that we can take the trivial extension of $\eta$ to $W_{G_0}(\Levi,\lambda)_\eta$. In the latter case, $\Levi$ is a $d$-split Levi subgroup of type $\tD_r(\pm q) (q^{d'}+1)^{a}$ of $\G$ and $W_G(\Levi,\lambda)\cong G(2d',2,a)$ and $W_{G_0}(\Levi,\lambda)\cong G(2d',1,a)$, see the proof of \cite[Thm.~ 4.6]{Ma07}.
    However, the irreducible characters of both groups have character values in $\mathbb{Q}_\ell$. Hence, every character extends to a character in $W_{G_0}(\Levi,\la)$ with values in $\mathbb{Q}_\ell$.
   
We now consider the field automorphisms. Let $\bT$ be a maximally $F$-split torus of the $d$-split Levi subgroup $\Levi$ of $\bG$. Let $\bT_0$ be a maximally split torus contained in an $F$-stable Borel subgroup of $\bG$ with Weyl group $W$ and associated set $\Delta$ of simple roots. Let $wF$ be the $F$-type of the maximal torus $\bT$ relative to $\bT_0$. In particular, there is $g \in \bG$ with $v:=\Lang(g)\in \Norm_G(\bT_0)$ and image $w$ in $W$ satisfying ${}^g \bT_0=\bT$. As $F_0$ acts trivially on $W$, we can assume that $F_0(v)=v.$ It follows that conjugation with $g$ induces an isomorphism $\bG^{F} \to \bG^{vF}$. In particular, we can consider the isomorphism $F_0: \bG^{vF} \to \bG^{vF}$ as an element of $(\bG^{vF} \langle F_0 \rangle)/ \langle vF \rangle$. With this notation, we have an isomorphism $\N_{G \langle F_0, \Gamma \rangle}(\Levi)/\Levi^F\cong \N_{(\bG \Gamma)^{vF} \langle F_0 \rangle}(\Levi_0)/\Levi_0^{vF}$ where $\Levi_0={}^g \Levi$. Denote by $W'$ the image of $W_G(\Levi,\la)$ under this isomorphism.

 As $F_0$ still centralizes $\N_{(\bG \Gamma)^{vF}}(\Levi_0)/ \Levi_0^{vF}$ and every character of $W_G(\Levi,\la)$ extends to its inertia group in $W_{G_0}(\Levi,\la)$, it follows that every character $\eta$ of $W_G(\Levi,\la)$ extends to $W_{\hat{G}}(\Levi,\la)_{\eta}$.

Let $d_0:=\lcm(\delta,d)$ and $r_0:=r/\delta$ so that $q=p^{r_0}$.  
As $\bT$ is of $F$-type $wF$ it follows that $\bT$ is of $F^{d_0}$-type $(wF)^{d_0}$. On the other hand, as $\Levi$ is a $d$-split Levi subgroup of $(\bG,F)$ it follows that $\Levi$ is a $1$-split Levi subgroup of $(\bG,F^{d_0})$. In particular, $\bT$ is a maximally split torus of the $1$-split Levi subgroup $\Levi$ of $(\bG,F^{d_0})$ and as such it is a maximally split torus of $(\bG,F^{d_0})$. 
As $F^{d_0}$ acts trivially on $W$ it thus follows that $(wF)^{d_0}=F^{d_0}$. Let $w_0\in W$ such that $w_0 F^{\delta}=(wF)^{\delta}$. 
Then $F_0^{r_0 \delta}=F^{\delta}=w_0^{-1}$ in $\N_{\bG^{vF} \langle F_0 \rangle}(\Levi_0)/\Levi_0^{vF}$. As $F^{\delta}$ acts trivially on $W$ and $(wF)^{d_0}=F^{d_0}$ we have $w_0^{d_0/\delta}=1$.
Hence, $F_0^{r_0 d_0}=1$ in $\N_{\bG^{vF} \langle F_0 \rangle}(\Levi_0)/\Levi_0^{vF}$. In particular, $F_0$ is an automorphism of order divisible by $r_0 d_0$ on $\N_{\bG^{vF} \langle F_0 \rangle}(\Levi_0)/\Levi_0^{vF}$. Now as $F_0$ acts trivially on $W$ it follows that $F_0$ acts trivially on $\N_{\bG^{vF} \langle F_0 \rangle}(\Levi_0)/\Levi_0^{vF}$. It follows that $\N_{\bG^{vF} \langle F_0 \rangle}(\Levi_0)/\Levi_0^{vF}$ is the central product of $\N_{\bG^{vF}}(\Levi_0)/\Levi_0^{vF}$ and $\langle F_0 \rangle$ over $\langle w_0 \rangle$.

If $\eta \in \Irr(W')$, then we have $\eta(w_0)=\eta(1) \zeta$ for some $d_0/\delta$th root of unity $\zeta$. Hence, every extension $\hat{\eta} \in \Irr(W' \langle F_0 \rangle \mid \eta)$ satisfies $\hat{\eta}(F_0)=\eta(1) \zeta_0$ where $\zeta_0^{r_0 \delta}=\zeta$. It therefore suffices to show that the polynomial $X^{r_0\delta}-\zeta_{d_0/\delta}$ has a zero in $\mathbb{Q}_\ell$, where $\zeta_{d_0/\delta}$ is a primitive $d_0/\delta$th root of unity. Assume first that $\delta=1$. In this case, by Lemma \ref{number theory} the polynomial $X^{r_0}-\zeta_d$ has a zero in $\mathbb{Q}_\ell$.  

Now if $\delta=2$ and $2 \mid d$ then we have $d_0/\delta=d/2$ and $X^{2r_0}-\zeta_{d/2}$ has a root in $\mathbb{Q}_\ell$ as $X^{r_0}-\zeta_{d}$ has a root in $\mathbb{Q}_\ell$ by what we have just proved. If $\delta=2$ and $d$ is odd it follows that $X^{2r_0}-\zeta_{d}$ has a root in $\mathbb{Q}_\ell$ as $X^{r_0}-\zeta_d$ has a root and $2 \nmid d$. The identical argument works in the case where $\delta=3$. 
\end{proof}

By \cite[Thm.~2.4]{MaExt} every unipotent character of $G$ extends to its inertia group in $\hat{G}$. On the other hand, by Lemma \ref{weyl value} every character of $W_G(\Levi,\la)$ extends to its inertia group in $W_{\hat{G}}(\Levi,\la)$. Hence, the result in \cite[Thm.~4.6]{Ma07} shows that there exists a bjection \[\Irr(W_{\hat{G}}(\Levi,\lambda)) \to \mathcal{E}(\hat{G}\mid(\Levi,\lambda)).\]
We want to show that this bijection lifts to an $\galh_\ell$-equivariant bijection \[\Irr(\mathrm{N}_{\hat{G}}(\Levi) \mid \lambda) \to \mathcal{E}(\hat{G} \mid (\Levi,\lambda) ).\] 

For this let us temporarily assume that $\G$ is of classical type.
By Lemma \ref{extensions local} and using that $\lambda$ is $\galh_\ell$-stable by Corollary \ref{cor:HCgalequiv}, there exists an $\galh_\ell$-invariant extension $\Lambda(\lambda)$ of $\lambda \in \Irr(L)$ to $\N_G(\Levi)_\lambda$. Recall that by $d$-Harish-Chandra theory and Clifford theory, we have a bijection 	$$\Omega:\Irr(\mathrm{N}_{G}(\Levi) \mid \lambda) \to \UCh(G\mid (\Levi,\lambda) ), \Ind_{\N_{G}(\Levi)_\lambda}^{\N_G(\Levi)}(\Lambda(\lambda) \eta) \mapsto R_\Levi^\bG(\lambda)_\eta,$$ 
where here $R_\Levi^\bG(\lambda)_\eta$ is the character in $\cE(G, (\Levi, \lambda))$ corresponding to $\eta$ under the bijection $\mathrm{I}_{(\Levi, \lambda)}^{(\bG)}$ discussed in Section \ref{sec:Ibijection}.
By Corollary \ref{cor:relweylfixed}, every character of $W_G(\Levi, \lambda)$ is $\galh_\ell$-invariant and the same is true for all unipotent characters of $G$ as $\G$ is of classical type. In particular, the bijection $\Omega$ is $\galh_\ell$-equivariant, as $\galh_\ell$ acts trivially on both sides. We generalize this statement to the characters of $\hat{G}$.

\begin{theorem}\label{thm:extunipclassical}
	Let $G$ be of classical type.
	For every unipotent $d$-cuspidal pair $(\Levi,\lambda)$, there exists a $\Irr(\hat{G}/G) \rtimes \galh_\ell$-equivariant bijection
	$$\Irr(\mathrm{N}_{\hat{G}}(\Levi) \mid \lambda) \to \mathcal{E}(\hat{G} \mid (\Levi,\lambda) ).$$ 
\end{theorem}

\begin{proof}
As in Proposition \ref{classical}, we consider an extension $\hat{\lambda} \in \Irr(\hat{N}_\lambda \mid \la)$. As argued in Proposition \ref{classical} for every $\chi=R_\Levi^{\bG}(\lambda)_\eta \in \mathcal{E}(G,(\Levi,\lambda))$ there exists an extension $\hat{\chi} \in \Irr(\hat{G}_\chi \mid \chi)$ such that $\mathbb{Q}_\ell(\hat{\chi})=\mathbb{Q}_\ell(\hat{\lambda})$. Note that by Lemma \ref{weyl value}, for each $\eta \in \Irr(W_G(\Levi,\lambda))$ there exists a character $\hat{\eta} \in \Irr(W_{\hat{G}}(\Levi,\lambda)_\eta)$ with character values in $\mathbb{Q}_\ell$. We define $$\hat{\Omega}:\Irr(\mathrm{N}_{\hat{G}}(\Levi) \mid \lambda) \to \mathcal{E}(\hat{G} \mid (\Levi,\lambda) )$$
by setting $\hat{\Omega}(\Ind_{\hat{N}_\lambda}^{\hat{N}}(\hat{\lambda} \hat{\eta})):=\Ind_{\hat{G}_\chi}^{\hat{G}}(\hat{\chi})$
and extending it $\Irr(\hat{N}/N)$-linearly. 
This map is $\galh_\ell$-equivariant as by Proposition \ref{classical} the action of $\sigma \in \galh_\ell$ on $\hat{\chi}$ and $\hat{\lambda}$ is the same. 
\end{proof}

\subsection{The bijection for groups of exceptional Lie type}\label{sec:exceptionalunipotentext}

In this section, we establish the corresponding result of Theorem \ref{thm:extunipclassical} for $G$ of exceptional type, in the situation of the McKay--Navarro conjecture.

\begin{theorem}\label{exceptional}
	Let $G$ be of exceptional type.
	For every unipotent $d$-cuspidal pair $(\Levi,\lambda)$ with $\Levi$ a minimal $d$-split Levi subgroup of $\bG$, there exists a $\Irr(\hat{G}/G) \rtimes \mathcal{H}_\ell$-equivariant bijection
	$$\hat{\Omega}: \Irr(\hat{N} \mid \Irr_{\ell'}(N \mid \la)) \to \mathcal{E}(\hat{G} \mid (\Levi,\lambda) ) \cap \Irr(\hat{G} \mid \Irr_{\ell'}(G)).$$ 
\end{theorem}

\begin{proof}
We can assume throughout the proof that $\Levi$ is a proper subgroup of $\bG$ since the claim is trivial otherwise.

We claim that there exists a bijection
$$\Omega:\Irr_{\ell'}(\mathrm{N}_{G}(\Levi) \mid \lambda) \to \mathcal{E}_{\ell'}(G, (\Levi,\lambda) ), \Ind_{\N_{G}(\Levi)_\lambda}^{\N_G(\Levi)}(\Lambda(\lambda) \eta) \mapsto R_\Levi^\bG(\lambda)_\eta.$$ 
For this, it suffices by Proposition \ref{prop:HCgalequiv} to show that the character $\la$ has a $\mathcal{H}_\ell$-invariant extension to $N_\la$. Note that whenever the relative Weyl group is non-cyclic $\Levi$ is a torus except when $\G=\tE_7(q)$ and $d=4$ (see \cite[Table 3.3]{GeckMalle}).
In this case $\Levi_0 \cong \PGL_2^3$ and thus every unipotent character of $L_0$ has a rational extension to its inertia group in $L_0 E(L_0)$. These considerations even show that $\la$ has a $\mathcal{H}_\ell$-invariant extension to $\hat{N}_\la$. A similar argument applies in the case where the relative Weyl group is cyclic. According to \cite[Table 3]{S09} the Levi subgroup $\Levi_0$ has in these cases only components of type $\tA_1$ and $\tA_2$. We also note that $\Levi_0$ has only a rational components of type $\tA_2(\delta q)$, $\delta \in \{ \pm 1\}$, if the order of $\delta q$ modulo $\ell$ is odd. Hence, again by Lemma \ref{lem:numtheory} every unipotent character of $L_0$ has an extension with values in $\mathbb{Q}_\ell$ to its inertia group in $L_0 E(L_0)$. 

In particular, Lemma \ref{weyl value} ensures that $\mathcal{H}_\ell$ acts trivially on $\Irr(\hat{N} \mid \Irr_{\ell'}(N \mid \la))$.

Arguing as in the proof of Theorem \ref{thm:extunipclassical} we are left to show that every character $ \chi \in \Irr_{\ell'}(G)$ has a $\mathcal{H}_\ell$-invariant character of $\hat{G}$ covering it. For this we use the information about unipotent characters in \cite[Section 13.9]{Carter}. It can now be checked that if $\omega_\chi$ is a $m$th root of unity then $m \mid d$. If $\omega_\chi=\sqrt{-q}$, then $d \in \{2,6,10,14,18,30\}$ so that in all of these cases $d'$, the order of $-q$ modulo $\ell$, is odd.
Applying Lemma \ref{number theory} to all these cases shows that $X^{r_0}- \omega_\chi$, where $q=p^{r_0}$, always has a zero in $\mathbb{Q}_\ell$. Hence, Lemma \ref{valuesF0} shows that every $\chi$ has a $\mathcal{H}_\ell$-invariant extension to $G \langle F_0 \rangle$. For graph automorphisms we apply Lemma \ref{values graph} and note that the exceptional characters in $\tE_6(\varepsilon q)$ whose extension have character field $\mathbb{Q}( \sqrt{\varepsilon q})$ only appear when $d$ is odd (resp. when $d'$ is odd). This shows by Lemma \ref{Gallagher} that every character $\chi$ has a $\mathcal{H}_\ell$-invariant character of $\hat{G}$ covering it. This completes the proof.
\end{proof}

Finally, we can now prove Theorem \ref{thm:unipcriterion} from the introduction, which says that the inductive McKay--Navarro conditions hold for $G$ and $\ell$ at unipotent characters.
 
\begin{proof}[Proof of Theorem \ref{thm:unipcriterion}]
As before let $\bG \hookrightarrow \tilde{\bG}$ be a regular embedding and observe that $\tilde{\bG} /\Z(\tilde{\bG})\cong\bG_{\mathrm{ad}}$, where $\bG_{\mathrm{ad}}$ is the adjoint group of the same type as $\bG$. Note that every unipotent character $\chi$ of $\bG^F$ extends to a unipotent character of $\tilde{\bG}^F$ with $\Z(\tilde{\bG})^F$ in its kernel. Its deflation can therefore be seen as a character of $\bG_{\mathrm{ad}}^F$. On the other hand, every unipotent character of $\bG_{\mathrm{ad}}^F$ has an extension $\hat{\chi}$ to its inertia group in $\bG_{\mathrm{ad}}^F \rtimes E(\bG_{\mathrm{ad}}^F)$. 
In particular, its inflation can be seen as a character of $\tilde{\bG}^F \rtimes E(\bG^F)_\chi$. On the other hand, the character $\Omega(\chi)$ lies in $\Irr(M \mid \la)$ where $(\Levi,\la)$ with $\Levi=\Cent_{\bG}(\mathbf{S})$ is the $d$-cuspidal pair in whose $d$-Harish-Chandra series $\chi$ lies. As $\la$ is again a unipotent character it follows again that $\la$ extends to a unipotent character of $(\Levi \Z(\tilde{\bG}))^F$ with $\Z(\tilde{\bG})^F$ in its kernel. Its deflation $\la_0$ is therefore a $d$-cuspidal character of $\Levi_{0}:=\Levi/\Z(\Levi)$. 
In particular, $\Omega(\chi)$ can be seen as a character in $\Irr_{\ell'}(\N_{\bG_{\mathrm{ad}}^F}(\Levi_0) \mid \la_0)$. Let $\hat\Omega$ be the map from Theorem \ref{thm:extunipclassical} or \ref{exceptional}. Then  we have $\hat{\psi}\in \Irr(\N_{\bG_{\mathrm{ad}}^F \rtimes E(\bG_{\mathrm{ad}}^F)}(\Levi_0)_\chi  \mid \Omega(\chi))$, where $\psi$ is the Clifford correspondent of $\hat{\Omega}(\Ind_{\bG_{\mathrm{ad}}^F \rtimes E(\bG_{\mathrm{ad}}^F)_\chi}^{\bG_{\mathrm{ad}}^F \rtimes E(\bG_{\mathrm{ad}}^F)}(\hat{\chi}))$. The representations associated to $\hat{\chi}$ and $\hat{\psi}$ now define projective representations (with trivial cocycle) associated to $\chi$ and $\psi$ respectively. Now as the map $\hat{\Omega}$ is equivariant with respect to linear characters and the action of $\mathcal{H}_\ell$ it follows that these representations satisfy all the properties in \cite[Def.~3.1]{NSV20}, or in other words, one has $((\widecheck G E)_{\chi^{\calH}}, G, \chi)_\cH \geq_c ((\widecheck{M}\wh M)_{\psi^\cH}, M, \psi)_{\cH}$.
\end{proof}

	\Addresses
	\printindex
	\appendix

\end{document}